\RequirePackage{fix-cm}
\documentclass[smallextended]{svjour3}       
\smartqed  
\usepackage{graphicx}
\usepackage{etex,bm}
\usepackage{amsmath,amssymb,mathrsfs,mathtools}
\usepackage{amsfonts}
\usepackage[margin=1in]{geometry}
\usepackage{multirow,booktabs}
\usepackage{empheq}
\usepackage{array}
\usepackage[colorlinks,citecolor=blue,urlcolor=blue,bookmarks=false,hypertexnames=true]{hyperref}

\begin{document}
\title{Robust Portfolio Optimization under Ambiguous Chance Constraints
}


\author{Pulak Swain         \and
        Akshay Kumar Ojha 
}


\author{Pulak Swain* \and Akshay Kumar Ojha }
\institute{Pulak Swain \at
	School of Basic Sciences, Indian Institute of Technology Bhubaneswar \\
	Tel.: +917978772457\\
	\email{ps28@iitbbs.ac.in}           
	\and
	Akshay Kumar Ojha \at
	School of Basic Sciences, Indian Institute of Technology Bhubaneswar
}

\date{Received: date / Accepted: date}

\maketitle

\begin{abstract}
In this paper, we discuss the ambiguous chance constrained based portfolio optimization problems, in which the perturbations associated with the input parameters are stochastic in nature, but their distributions are not known precisely. We consider two different families of perturbation distributions-- one is when only upper $\&$ lower bounds of mean values are known to us, and the second one is along with the mean bounds, we also have the knowledge of the standard deviations of the perturbations. We derive the safe convex approximations of such chance constrained portfolio problems by using some suitable generating functions such as a piecewise linear function, an exponential function, and a piecewise quadratic function. These safe approximations are the robust counterparts to our ambiguous chance constrained problem and they are computationally tractable due to the convex nature of these approximations. 
\end{abstract}

\section{Introduction} \label{section 1}
\par The problem of optimal asset allocation is an important aspect of investment studies and it has been truly identified by Harry Markowitz in 1952 \cite{Markowitz J 1952}. Markowitz's studies show the importance of diversification of investment through the optimization models \cite{Markowitz 1959}. In such problems, the prime objective of the investor is to find the right balance between two conflicting components- risk and return. To visualize this opposing behavior of the above two components, a trade-off curve can be drawn which is known as the efficient frontier. Markowitz's mean-variance portfolio model is defined by using the expected returns and covariance of returns of the individual assets as the input parameters. For a mathematical overview, let's assume that the investor wants to invest his total wealth among $n$ assets where the return of any arbitrary $i^{th}$ asset over a time period $t$ is given by $r_{it} \ (t = 1, 2, \dots, T)$. Let $\mu_i$ and $\sigma_{ij}$ are respectively the expected return of $i^{th}$ asset and the covariance of return between $i$-$j$ pair of assets, which are defined as,
\begin{align*}
\begin{aligned}
&\mu_i=\dfrac{1}{T} \sum_{t=1}^{T} r_{it}, \quad \sigma_{ij}=\dfrac{1}{T} \sum_{t=1}^{T} (r_{it}-\mu_i) (r_{jt}-\mu_j)
\end{aligned}
\end{align*}
Let $x_i$ be the weight allocated for $i^{th}$ asset in the portfolio. Then the expected portfolio return ($\mu_p$) and the variance of portfolio return ($\sigma_p^2$) can be calculated as,
\begin{align*}
\begin{aligned}
&\mu_p=\sum_{i=1}^{n} \mu_i x_i, \quad \sigma_{p}^2=\sum_{i=1}^{n} \sum_{j=1}^{n} \sigma_{ij} x_i x_j
\end{aligned}
\end{align*}
Markowitz portfolio model minimizes the variance of portfolio return by constraining the expected return to become more than or equal to a target return ($\tau$). Moreover, the model assumes the weights to be non-negative and the sum of the weights to be $1$. Mathematically the Markowitz portfolio model is given by,
\begin{align}
\begin{aligned}	\label{eq*}
&\min_{x_i} && \frac{1}{2} \sum_{i=1}^{n} \sum_{j=1}^{n} \sigma_{ij} x_i x_j \\
&\textrm{s.t.:} && \sum_{j=1}^{n} \mu_j x_j \geq 1, \quad \sum_{j=1}^{n} {x_j}=1, \quad {x_j}\geq 0, \ j=1, 2, \dots, n.
\end{aligned}
\end{align}
\par One of the biggest challenges in portfolio optimization is to get the exact input parameters such as expected asset returns and the covariance of asset returns. The main reason for this is the several economic factors involved in the market which can affect these input parameters. This leads to the presence of some uncertainty factor in these parameters, which if ignored can potentially affect the optimal asset allocation. To address this issue, several approaches such as sensitivity analysis, dynamic programming, fuzzy programming, etc. are used in the field of optimization. However, in the last two decades, robust optimization approach has been widely used by researchers due to its ability to produce "completely immunized against uncertainty" solutions. This approach was first studied by El Ghaoui and Lebret \cite{Ghaoui 1997}, in which they worked on the uncertain least square problems. Then El Ghaoui et al. \cite{Ghaoui 1998} used the robust optimization approach for uncertain semidefinite programs and provided sufficient conditions for the robust solutions. Later on Ben-Tal and Nemirovski \cite{Ben-Tal 1999} studied the robust optimization for linear programming problems for ellipsoidal uncertainty and showed the tractability of the robust counterpart. Another work of Ben-Tal and Nemirovski \cite{Ben-Tal 2000} was based on some uncertain linear programming problems from NETLIB collection and their results showed that the robust solutions nearly lose nothing in optimality. In the literature, there are some more studies on the geometrical shape based uncertainties such as box, ellipsoidal, polyhedral, etc. The robust optimization approach was first used in the uncertain portfolio problems when Goldfarb and Iyenger \cite{Goldfarb 2003} introduced some uncertain structures for the market parameters and reformulated the robust portfolio problems into second order cone programming. T\"{u}t\"{u}nc\"{u} and Koenig \cite{Tutuncu 2004} formulated the robust counterparts when the expected return vector and the covariance matrix of asset returns are defined by lower and upper bounds and they also showed the robust efficient frontiers. Fliege and Werner \cite{Fliege 2014} studied the application of robust multiobjective optimization on uncertain mean-variance portfolio problems. Sehgal and Mehra \cite{Sehgal 2019} proposed robust portfolio optimization models for reward–risk ratios
by using Omega, semi-mean absolute deviation ratio, and weighted stable tail adjusted return ratio and also showed the efficiency of these models.
\par On the other hand, the chance constrained based optimization problems \cite{Biswal 1998} is another type of uncertain optimization problem in which the uncertain parameters are stochastic in nature. Moreover, chance constraints basically indicate the probability of uncertain constraint(s) to be satisfied, where the perturbations act as random variables. Calafiore and El Ghaoui \cite{Calafiore 2006} studied distributional robust chance constrained linear problems with radially distributed perturbations. They showed that the linear chance constraints can be converted explicitly into convex second-order cone constraints so that they can be efficiently solvable. Erdoğan and Iyengar \cite{Erdogan 2006} studied the ambiguous chance constrained problem which is approximated by a robust sampled problem. Nemirovski and Shapiro \cite{Nemirovski 2007} studied the convex approximations for the ambiguous chance constraints. In particular, they derived the Bernstein approximation and their approximation was convex and efficiently solvable. Yanıkoğlu and den Hertog \cite{Yanıkoğlu 2013} used the available historical data of uncertain parameters to get the safe approximations of ambiguous chance constraints and their approach has the advantage that it can also easily handle joint chance constraints. Bertsimas et al. \cite{Bertsimas 2018} derived safe approximations to nonlinear robust individual chance constraints by using statistical hypothesis tests. Zhang et al. \cite{Zhang 2018} solved distributionally robust chance-constrained binary programs by equivalently reformulating them as $0-1$ second order cone programs.
\par The main contribution of the paper is as follows: Sometimes the uncertainty present in the portfolio optimization problems is stochastic in nature and we don't know the exact distribution of the uncertain parameters. Such problems are in general difficult to solve, as in those problems we not only need to handle the chance constraints, but also we need to take care of the ambiguous nature of probability distributions present in the chance constraints. So the challenge is to convert such problems into deterministic as well as computationally tractable optimization problems. In this paper we use some generating functions to get the safe convex approximations to such problems. We consider two different families of perturbation distributions-- (i) when we know the bounds of mean perturbation values and (ii) when we know the mean and standard deviations of perturbation values. The safe convex approximations which we get for the family-i are quadratic programming problems and the approximations we get for the family-ii are quadratically constrained quadratic programming problems, which are computationally tractable. Moreover, we implemented the results into a real life stock market problem to find out the robust optimal allocation. Also, the risk-return efficient frontiers are drawn for all the safe approximation cases.
\par The remainder of the paper is organized as follows: Section \ref{section2} presents some definitions and basic concepts regarding robust optimization and the chance constrained based portfolio optimization problems. Section \ref{section3} constitutes the discussion on generating function based safe convex approximation and then the safe approximations for several families of perturbation distributions are derived in \ref{section4}. Further, a numerical example consisting of an Indian stock market problem is given in Section \ref{section5}. Finally, some concluding remarks are provided in Section \ref{section6}.
\section{Preliminaries} \label{section2}
\subsection{Uncertainty in Optimization Problems and the Robust Optimization Approach}
When the constant coefficients in an optimization problem are not fixed values; but they perturb around some nominal values, such problems are called uncertain optimization problems. Mathematically this can be represented as \cite{Bertsimas 2011},
\begin{align} \label{eq0}
	\begin{aligned}	
		&\min_{\bm x} && f(\bm{x},\bm{u}) \\
		&\textrm{s.t.:} && c(\bm{x},\bm{u}) \geq 0, \quad \forall \bm u = \bm{u}^{(0)}+\displaystyle{\sum_{j=1}^{n}}\zeta_j \bm{u}^{(j)}  \in \mathscr{U} 
	\end{aligned}
\end{align}
where $\bm x$ is the vector of decision variables and $\bm u$ is the vector of uncertain variables, $\bm{u}^{(0)}$ is the vector of nominal values of uncertain parameters and $\bm{u}^{(j)}$ (for $j=1, 2, \dots, n$) are the basic shifts.\\
Now we will go through some definitions related to robust optimization.
\begin{definition} [Robust Feasible Solution] \label{def1} \cite{Ben-Tal 2009}
	A vector $\bm x$ is said to be the robust feasible solution of the uncertain problem \eqref{eq0} if it satisfies the uncertain constraints for all realizations of the uncertain set $\mathscr{U}$, that is, if $\bm x$ satisfies
	\begin{align*}
		\begin{aligned}	
			c(\bm{x},\bm{u}) \leq 0, \quad \forall \bm u  \in \mathscr{U} 
		\end{aligned}
	\end{align*}
\end{definition}
\begin{definition} [Robust Value] \label{def2} \cite{Ben-Tal 2009}
For a given candidate solution $\bm{x}$, the robust value $\widehat{f}(\bm{x})$ of the objective in problem \eqref{eq0} is the largest value of $f(\bm{x},\bm{u})$ over all realizations of the data from the uncertain set, that is
	\begin{align*}
		\begin{aligned}	
			\widehat{f}(\bm{x})= \sup_{\bm u  \in \mathscr{U}} f(\bm{x},\bm{u})
		\end{aligned}
	\end{align*}
\end{definition}
It is to be noted that if the problem \eqref{eq0} were a maximization problem, then the robust value would have been $\widehat{f}(\bm{x})= \inf_{\bm u  \in \mathscr{U}} f(\bm{x},\bm{u})$.
\begin{definition} [Robust Counterpart] \label{def3} \cite{Ben-Tal 2009}
	The robust counterpart (RC) of the uncertain problem \eqref{eq0} is the optimization problem 
	\begin{align*}	
		\begin{aligned} 
			& \min_{\bm x} && \left \{ \sup_{\bm u \in \mathscr{U}}	f(\bm x, \bm u) \right \}\\
			&\text{s.t.:} && c(\bm x, \bm u) \leq \bm 0, \quad \forall \bm u  \in \mathscr{U}
		\end{aligned}
	\end{align*}
	of minimizing the robust value of the objective over all the robust feasible solutions to the uncertain problem.
\end{definition}
\begin{definition} [Robust Optimal Solution] \label{def4} \cite{Ben-Tal 2009}
	The solution of the RC problem is said to be the robust optimal solution of the uncertain problem \eqref{eq0}.
\end{definition}
\subsection{Ambiguous Chance Constraint}
Consider the uncertain constraint given in the problem \eqref{eq0},
\begin{align}
	\begin{aligned}	\label{eq1}
		c(\bm x, \bm u) \geq 0, \quad \bm{u}= \bm{u}^{(0)}+\displaystyle{\sum_{j=1}^{n}}\zeta_j \bm{u}^{(j)}
	\end{aligned}
\end{align}
In the ideal scenario, we would like to find a solution $\bm x$ which satisfies the constraint \eqref{eq1} for all realizations of the uncertain set. But when the data perturbation is stochastic in nature, it is not always possible to find such a solution. So we look for a candidate solution that satisfies \eqref{eq1} for "nearly all" realizations of $\bm{\zeta}$. Thus we want our constraint to be satisfied with a predefined reliability level. Then the probability chance constraint to \eqref{eq1} is introduced which is given by, 
\begin{align}
\begin{aligned}	\label{eq2}
{Prob}_{\bm{\zeta}\sim P} \left \{\bm{\zeta}: c(\bm x, \bm u) \geq 0, \quad \bm{u}= \bm{u}^{(0)}+\displaystyle{\sum_{j=1}^{n}}\zeta_j \bm{u}^{(j)} \right \} \geq \beta
\end{aligned}
\end{align}
Eq. \eqref{eq2} indicates that the constraint \eqref{eq1} would be satisfied with a probability of at least $\beta$, where $\beta$ may be taken as close to 1.\\
However, sometimes the probability distribution of $\bm{\zeta}$ is not known precisely; all we know is that the distribution belongs to a given family $\mathcal{P}$ of probability distributions on $\mathbb{R}^n$. In such cases, we need to convert the problem \eqref{eq2} into an ambiguously chance
constrained version which is given by,
\begin{align*}
\begin{aligned}	
{Prob}_{\bm{\zeta}\sim P} \left \{\bm{\zeta}: c(\bm x, \bm u) \geq 0, \quad \bm{u}= \bm{u}^{(0)}+\displaystyle{\sum_{j=1}^{n}}\zeta_j \bm{u}^{(j)} \right \} \geq \beta, \quad \forall P \in \mathcal{P}
\end{aligned}
\end{align*}
One of the biggest drawbacks of these ambiguous chance constrained problems is the tractability issues of the robust counterparts. The problem becomes NP-hard even for the cases when P is simple.
\subsection{Uncertain Portfolio Problem with Chance Constraint}
Consider the portfolio optimization problem \eqref{eq*} in vector form given by,
\begin{align}
\begin{aligned}	\label{eq2a}
&\min_{\bm x} && \frac{1}{2} \bm{x}^\top\bm{\Sigma} \bm{x} \\
&\textrm{s.t.:} && \bm{\mu}^\top \bm{x} \geq \tau, \quad	\bm{e}^\top\bm{x}=1, \quad \bm{x}\geq \bm{0}
\end{aligned}
\end{align}
where $\bm{\mu}=\begin{bmatrix}
\mu_1\\  \mu_2\\ \vdots\\ \mu_n
\end{bmatrix}$, $\bm{\Sigma}=\begin{bmatrix}
\sigma_{11} &  \sigma_{12} & \dots & \sigma_{1n}\\
\sigma_{21} &  \sigma_{22} & \dots & \sigma_{2n}\\
\vdots & \vdots & \ddots & \vdots\\
\sigma_{n1} &  \sigma_{n2} & \dots & \sigma_{nn}\\
\end{bmatrix}$,
 $\bm{x}=\begin{bmatrix} x_1\\  x_2\\ \vdots\\ x_n \end{bmatrix}$,
 $\bm e=\begin{bmatrix} 1\\  1\\ \vdots\\ 1 \end{bmatrix}$,
 $\bm 0=\begin{bmatrix} 0\\  0\\ \vdots\\ 0 \end{bmatrix}$.\\
The uncertainty in portfolio optimization problems in general occurs due to the input parameters -- expected returns and covariance of returns of the assets. However, the studies show that the uncertainty in the covariance of asset returns does not affect the optimal solution so much \cite{Pulak 2021}. So in this study, we assume that the uncertainty occurs only in the expected returns of assets.\\
 Then the uncertain portfolio model can be written as,
\begin{align}
\begin{aligned} \label{eq3a}	
	&\min_{\bm x} && \frac{1}{2} \bm{x}^\top\bm{\Sigma} \bm{x} \\
	&\textrm{s.t.:} && \bm{\mu}^\top \bm{x} \geq \tau, \quad \bm{\mu} \in \mathscr{U}_{\bm{\mu}},\\
	& && \bm{e}^\top\bm{x}=1, \quad \bm{x}\geq \bm{0}
\end{aligned}
\end{align}
where $\mathscr{U}_{\bm{\mu}}$ is the uncertainty set associated with $\bm{\mu}$ given by,
\begin{align*}
\begin{aligned}	
	\mathscr{U}_{\mu}=\left\{\bm{\mu} : \bm{\mu} \equiv \bm{\mu}(\bm{\zeta}) = \bm{\mu}^{(0)}+\sum_{j=1}^{n}\zeta_j \bm{\mu}^{(j)} \right\}
\end{aligned}
\end{align*}
Now the model \eqref{eq3a} can be written with ambiguous chance constraint as,
\begin{align}
\begin{aligned}	\label{eq4}
	&\min_{\bm x} && \frac{1}{2} \bm{x}^\top\bm{\Sigma} \bm{x} \\
	&\textrm{s.t.:} && \forall P \in \mathcal{P} : {Prob}_{\bm{\zeta}\sim P} \left \{[\bm{\mu}(\bm{\zeta})]^\top \bm{x} \geq \tau \right \} \geq \beta, \\
	& &&	\bm{e}^\top\bm{x}=1, \quad \bm{x}\geq \bm{0}
\end{aligned}
\end{align}
The chance constraint in problem \eqref{eq4} can also be written as,
\begin{align}
\begin{aligned}	\label{eq5}
& {Prob}_{\bm{\zeta}\sim P} \left \{[\bm{\mu}(\bm{\zeta})]^\top \bm x <\tau \right \} \leq 1-\beta
\end{aligned}
\end{align}

\section{Generating-Function-Based Safe Convex Approximation for Ambiguous Chance Constrained Portfolio Optimization Model} \label{section3}
The major drawbacks of chance constrained problems are:
\begin{itemize}
\item It is difficult to evaluate with high accuracy the probability on the left side of a chance constraint.\\
\item In most of the cases the feasible set of the chance constraint is non-convex, due to which the optimization problem becomes very difficult to solve.
\end{itemize}
One way to overcome these issues is to replace the chance constraint in the model \eqref{eq4} with its computationally tractable safe approximation and then solve that approximation to get the robust solution. 
\begin{definition} [Safe Convex Approximation] \label{def5} \cite{Ben-Tal 2009}
Let $\{\bm{\mu}^{(j)}\}_{j=0}^{n}$, $Prob$, $\beta$ be the data of chance constraint \eqref{eq5} and let $S$ be a system of convex constraints on $\bm x$ and additional variables $\bm{v}$. Then $S$ is said to be a safe convex approximation of chance constraint \eqref{eq5} if the $\bm{x}$ component of every feasible solution $(\bm{x}; \bm{v})$ of $S$ is feasible for the chance constraint.
\end{definition}
Moreover, a safe convex approximation $S$ of \eqref{eq5} is said to be computationally tractable, if the convex constraints forming $S$ are efficiently computable.\\
Nemirovski and Shapiro \cite{Nemirovski 2007} suggested a generating function based approach to obtain the computationally tractable safe convex
approximation, where the generating function $\gamma(\cdot)$ must satisfy the
following properties:\\
(i) $\gamma(\cdot): \mathbb R \to \mathbb R$ is a non-negative non-decreasing function.\\
(ii) $\gamma(\cdot)$ is convex.\\
(iii) $\gamma(0) \geq 1$ and $\gamma(t) \to 0$ as $t \to -\infty$.\\
Let's denote the left hand side of the chance constraint \eqref{eq5} by $p(\bm{\mu}, \bm{x})$ and the constraint is given by,
\begin{align}
\begin{aligned}	\label{eq6}
& p(\bm{\mu}, \bm{x}) \equiv {Prob}_{\bm{\zeta}\sim P} \left \{[\bm{\mu}(\bm{\zeta})]^\top \bm x <\tau \right \} \leq 1-\beta
\end{aligned}
\end{align}
The constraint can also be written as,
\begin{align}
\begin{aligned} \label{eq6a}
	& E \left [\chi(\tau-\bm{\mu}(\bm{\zeta})^\top \bm x)\right ] \leq 1-\beta,
\end{aligned}
\end{align}
where $\chi(t)$ is the characteristic function given  by,
\begin{align*}
\begin{aligned}	
&\chi(t)=\left \{ \begin{array}{ll}0 &\mbox{ for } t \leq 0\\
1 &\mbox{ for } t > 0 \end{array} \right.
\end{aligned}
\end{align*} 
and $E[.]$ is the expectation of the function defined as,
\begin{align*}
\begin{aligned}	
&  E \left [\chi(\tau-\bm{\mu}(\bm{\zeta})^\top \bm x)\right ]= \int \chi(\tau-\bm{\mu}(\bm{\zeta})^\top \bm x) \ dP(\bm{\zeta}).
\end{aligned}
\end{align*}
Since $\chi(\cdot)$ is not a convex function, it follows that we can not convert \eqref{eq6a} into a convex constraint of $\bm{\mu}$ by using the function $\chi(\cdot)$.\\
To find out a convex approximation for the chance constraint \eqref{eq6}, let $\gamma(t)$ be a convex function on the axis which is everywhere $\geq$ $\chi(t)$ and it satisfies the properties (i)-(iii) mentioned before.\\
Then we have,
\begin{align}
\begin{aligned}	\label{eq7}
& p(\bm{\mu}, \bm{x})=E \left [\chi(\tau-\bm{\mu}(\bm{\zeta})^\top \bm x)\right ] \leq E \left [\gamma(\tau-\bm{\mu}(\bm{\zeta})^\top \bm x)\right ]=\psi(\bm{\mu}, \bm{x}) \text{ (say)}
\end{aligned}
\end{align}
Now let $\psi^+(\bm{\mu}, \bm{x})$ is some upper bound of $\psi(\bm{\mu}, \bm{x})$, that is, $\psi(\bm{\mu}, \bm{x}) \leq \psi^+(\bm{\mu}, \bm{x}) ~ \forall \bm{\mu}$. Then the constraint 
\begin{align}
\begin{aligned}	\label{eq7a}
& \psi^+(\bm{\mu}, \bm{x}) \leq 1-\beta
\end{aligned}
\end{align} 
can be a safe convex approximation of the chance constraint \eqref{eq5}.\\
When the function $\psi^+(\cdot)$ is efficiently computable, the safe approximation \eqref{eq7a} is tractable.\\
This observation can be converted into the following result:
 \begin{theorem} \label{prop1}
 	For the ambiguous chance constraint \eqref{eq5} of an uncertain portfolio optimization problem \eqref{eq4}, a generating function $\gamma(\cdot)$ can be used to write the safe convex approximation given by \eqref{eq7a}, where $\psi^+(\bm{\mu}, \bm{x}) \geq E \left [\gamma(\tau-\bm{\mu}(\bm{\zeta})^\top \bm x)\right ]$. Moreover, this approximation is tractable, provided that $\psi^+(\cdot)$ is efficiently computable.
 \end{theorem}
It is to be noted that, to reduce the conservatism of the above approximation, we can also do scaling as follows:\\
From the definition of $\chi(t)$, it can be observed that, $\chi(t)=\chi(\alpha ^{-1}t)$, $\forall \alpha >0$, which implies $\gamma(\alpha ^{-1}t) \geq \chi(\alpha ^{-1}t) = \chi(t)$, $\forall t$.\\
So from eq. \eqref{eq5} we get,
\begin{align}
\begin{aligned}	\label{eq8}
& p(\bm{\mu}, \bm{x})\geq \psi^+(\alpha^{-1}\bm{\mu}, \bm{x})
\end{aligned}
\end{align}
and the constraint 
\begin{align}
\begin{aligned}	\label{eq9}
& \psi^+(\alpha^{-1}\bm{\mu}, \bm{x}) \leq 1-\beta
\end{aligned}
\end{align}
is also a safe convex approximation for all $\alpha>0$.\\
This implies $\alpha\psi^+(\alpha^{-1}\bm{\mu}, \bm{x}) - \alpha(1-\beta) \leq 0$ is a safe approximation for any $\alpha>0$.\\
So the infimum of the left hand side of the function can also be used to write the safe convex approximation, which is given by,
\begin{align}
\begin{aligned} \label{eq10}
\inf_{\alpha >0} \left [\alpha\psi^+(\alpha^{-1}\bm{\mu}, \bm{x})-\alpha(1-\beta) \right ] \leq 0
\end{aligned}
\end{align}
When the function $\psi^+(\cdot)$ is efficiently computable, the left hand side of \eqref{eq10} is also efficiently computable. Then the approximation \eqref{eq10} is tractable.\\
In the next section, we obtain the robust counterparts of the problem \eqref{eq4} for several families $\mathcal{P}$ of perturbation distributions.
\section{Safe Convex Approximations for Several Families of Perturbation Distributions} \label{section4}
We consider two different families of perturbation distributions-- (i) the family whose upper $\&$ lower mean bounds are only known to us, and (ii) the family whose upper $\&$ lower mean bounds, and the standard deviations are known to us.\\
To derive the safe approximations we use some generating functions $\gamma(t)$ which are greater or equal to $\chi(t)$ everywhere and satisfy the properties given earlier.
\subsection{Family of Perturbation Distributions with Known Upper $\&$ Lower Mean Bounds} \label{sec4.1}
In this case we consider the family of $\bm{\zeta}=[\zeta_1, \zeta_2, \dots, \zeta_n]^\top$ with the upper $\&$ lower bounds of each $\zeta_j$, $(j=1, 2, \dots, n)$ are known to be $m_j^U$ $\&$ $m_j^L$ respectively. For deriving the safe convex approximation we use a piecewise linear function and the exponential function as the generating functions.
\subsubsection{Piecewise Linear Generating Function Based Approximation}
	The function
\begin{align}
\begin{aligned}	\label{eq10a}
	&\gamma(t)=[1+t]_+=\left \{ \begin{array}{ll}0 &\mbox{ for } t \leq -1\\
		1+t &\mbox{ for } t > -1 \end{array} \right.
\end{aligned}
\end{align}
is everywhere greater than or equal to $\chi(t)$. Further, it
is a non-negative and  non-decreasing convex function satisfying $\gamma(0) \geq 1$ and $\lim_{t \to -\infty}\gamma(t) =0$ and hence can be considered as a generating function.\\
\begin{figure}[h!]
	\centering
\includegraphics[height=5cm,width=12cm]{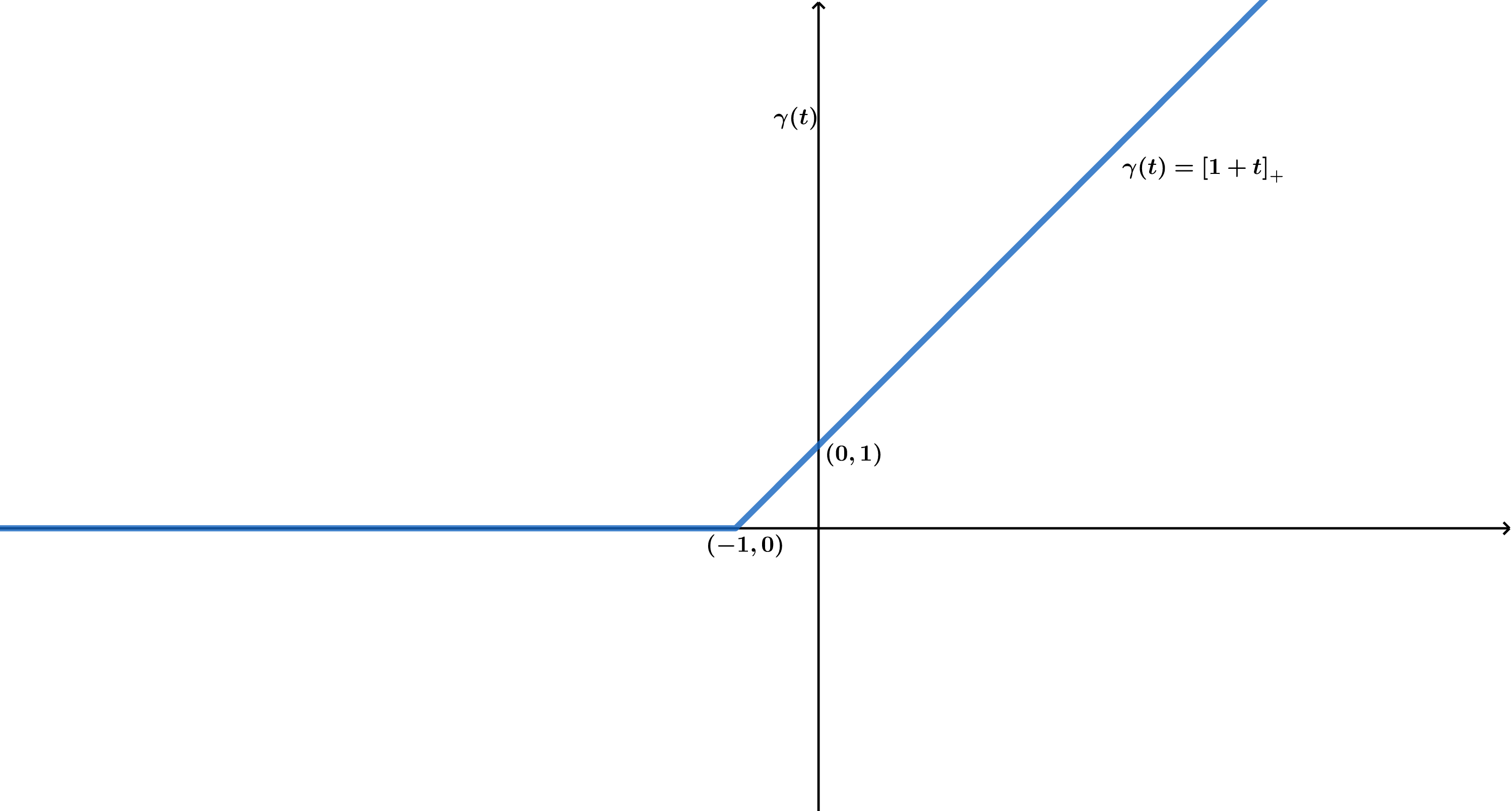}
	\caption{Piecewise Linear Generating Function}
	\label{fig1}
\end{figure}
Now since $E \left [\chi(\tau-\bm{\mu}(\bm{\zeta})^\top \bm x)\right ] \leq E \left [\gamma(\tau-\bm{\mu}(\bm{\zeta})^\top \bm x)\right ]$ by eq. \eqref{eq7}, it follows that we can replace the chance constraint \eqref{eq6a} by the constraint given by,
\begin{align}
\begin{aligned} \label{eq10b}
	& \qquad E_{\bm{\zeta \sim P}} \left [ \gamma(\tau-\bm{\mu}(\bm{\zeta})^\top x) \right ] \leq 1-\beta 
\end{aligned}
\end{align}
By using the definition of $\gamma(.)$ given by \eqref{eq10a}, the constraint \eqref{eq10b} can be written as,
\begin{align}
\begin{aligned} \label{eq10c}
 &  \qquad E_{\bm{\zeta \sim P}}  \left [\max\left \{0, 1+\tau-\bm{\mu}(\bm{\zeta})^\top x \right \}\right ]  \leq 1-\beta
\end{aligned}
\end{align}
The constraint \eqref{eq10c} can be replaced with two equivalent constraints,
\begin{subequations}
	\begin{align}
	 & E_{\bm{\zeta \sim P}}  \left [0 \right ]  \leq 1-\beta \label{eq10di}\\
	& E_{\bm{\zeta \sim P}}  \left [1+\tau-\bm{\mu}(\bm{\zeta})^\top x \right ]  \leq 1-\beta \label{eq10dii}
\end{align}
\end{subequations}
Here the constraint \eqref{eq10di} is trivial, as by the assumption of chance constraint problems the value of confidence level $\beta$ lies between $0$ and $1$, and hence we have $1-\beta \geq 0$. Therefore now we only need to proceed with the constraint \eqref{eq10dii}.\\
The left hand side of \eqref{eq10dii} can be simplified in component form as,
\begin{align}
\begin{aligned} \label{eq10e}
E_{\bm{\zeta \sim P}}  \left [1+\tau-\bm{\mu}(\bm{\zeta})^\top x \right ] &=E_{{\zeta_j \sim P_j}} \left [1+\tau-\sum_{j=1}^{n} {\mu_j^{(0)} x_j}-\sum_{j=1}^n {(\mu_j^{(j)} x_j)\zeta_j}\right ]\\
&=  1+\tau-\sum_{j=1}^{n} {\mu_j^{(0)} x_j}- E_{{\zeta_j \sim P_j}} \left [\sum_{j=1}^n {(\mu_j^{(j)} x_j)\zeta_j}\right ]\\
&=  1+\tau-\sum_{j=1}^{n} {\mu_j^{(0)} x_j}- \sum_{j=1}^n {(\mu_j^{(j)} x_j)E_{{\zeta_j \sim P_j}}\left [\zeta_j \right ]}
\end{aligned}
\end{align}
It is to be noted that $\zeta_j$'s are the random variables in eq. \eqref{eq10e}, so the terms without $\zeta_j$ are treated as constants.\\
Now our objective is to find out an upper bound of the right hand side of eq. \eqref{eq10e}, which would be our $\psi^+(\bm{\mu}, \bm{x})$ given in Theorem \ref{prop1}. Since the upper $\&$ lower mean bounds each perturbation $\zeta_j$ in the family are known to be $m_j^U$ $\&$ $m_j^L$ respectively, that is, $m_j^L \leq E_{{\zeta_j \sim P_j}}\left [\zeta_j \right ] \leq m_j^U$, it follows from eq. \eqref{eq10e} that,
\begin{align}
\begin{aligned} \label{eq10f}
	E_{\bm{\zeta \sim P}}  \left [1+\tau-\bm{\mu}(\bm{\zeta})^\top x \right ] \leq  1+\tau-\sum_{j=1}^{n} {\mu_j^{(0)} x_j}- \sum_{j=1}^n {(\mu_j^{(j)} x_j)m_j^L}
\end{aligned}
\end{align}
Thus by taking $\psi^+(\bm{\mu}, \bm{x})=1+\tau-\sum_{j=1}^{n} {\mu_j^{(0)} x_j}- \sum_{j=1}^n {(\mu_j^{(j)} x_j)m_j^L}$, the condition in Theorem \ref{prop1} is satisfied and the constraint
\begin{align*}
\begin{aligned}	
& 1+\tau-\sum_{j=1}^{n} {\mu_j^{(0)} x_j}- \sum_{j=1}^n {(\mu_j^{(j)} x_j)m_j^L} \leq 1-\beta
\end{aligned}
\end{align*}
can be replaced with the chance constraint.\\
Therefore a safe approximation of the robust counterpart to the portfolio model \eqref{eq4} is given by,
\begin{align}
\begin{aligned}	\label{eq10g}
	&\min_{x_i} && \frac{1}{2} \sum_{i=1}^{n} \sum_{j=1}^{n} \sigma_{ij} x_i x_j \\
	&\textrm{s.t.:} && 1+\tau-\sum_{j=1}^{n} {\mu_j^{(0)} x_j}- \sum_{j=1}^n {(\mu_j^{(j)} x_j)m_j^L} \leq 1-\beta,\\
	& && \sum_{j=1}^n {x_j}=1, \quad {x_j}\geq 0, \ j=1, 2, \dots, n.
\end{aligned}
\end{align}
The robust counterpart model \eqref{eq10g} is a quadratic programming problem, so it is computationally tractable.\\
The above result can be represented as the following theorem:
\begin{theorem} \label{theorem1}
Consider the ambiguous chance constrained portfolio optimization model \eqref{eq4}, in which the probability distribution of the perturbations $\zeta_j$, $(j=1, 2, \dots, n)$ are not known precisely, and the only information regarding each $\zeta_j$ is that it belongs to a family of perturbations having known upper $\&$ lower mean bounds $m_j^U$ $\&$ $m_j^L$ respectively. Then a safe tractable convex approximation to this model can be a quadratic programming problem given by,
\begin{align*}
\begin{aligned}	
&\min_{x_i} && \frac{1}{2} \sum_{i=1}^{n} \sum_{j=1}^{n} \sigma_{ij} x_i x_j \\
&\textrm{s.t.:} && 1+\tau-\sum_{j=1}^{n} {\mu_j^{(0)} x_j}- \sum_{j=1}^n {(\mu_j^{(j)} x_j)m_j^L} \leq 1-\beta,\\
& && \sum_{j=1}^n {x_j}=1, \quad {x_j}\geq 0, \ j=1, 2, \dots, n.
\end{aligned}
\end{align*}
\end{theorem}
\subsubsection{Bernstein Approximation}
One of the first studies of the Bernstein approximation scheme was done by Nemirovski and Shapiro \cite{Nemirovski 2007}. The scheme has two assumptions:\\
(i) $\mathcal{P}$ can be written as,
\begin{align*}
	\mathcal{P}= \left \{ P_1 \times P_2 \times \dots \times  P_n: P_j \in \mathcal{P}_j, \  j= 1, 2, \dots, n \right \}
\end{align*}
where for $j= 1, 2, \dots, n$, the distributions ${P}_j$ are independent and belong to the families $\mathcal{P}_j$ of distributions.\\
(ii) The moment generating functions $M_j(t)=E[e^{t \zeta_j}]=\int e^{tz} dP_j(z)$ $(j=1, 2, \dots, n)$ are finite valued for all $t \in \mathbb{R}$ and are efficiently computable.\\
The generating function for this approximation is given by,
\begin{align}
\begin{aligned} \label{eq10h}
\gamma(t)=e^t
\end{aligned}
\end{align}
which is everywhere greater than or equal to $\chi(t)$. Further, it can be easily verified that $e^t$ satisfies the properties of generating function.\\
\begin{figure}[!htb]
	\centering
	\includegraphics[height=5cm,width=12cm]{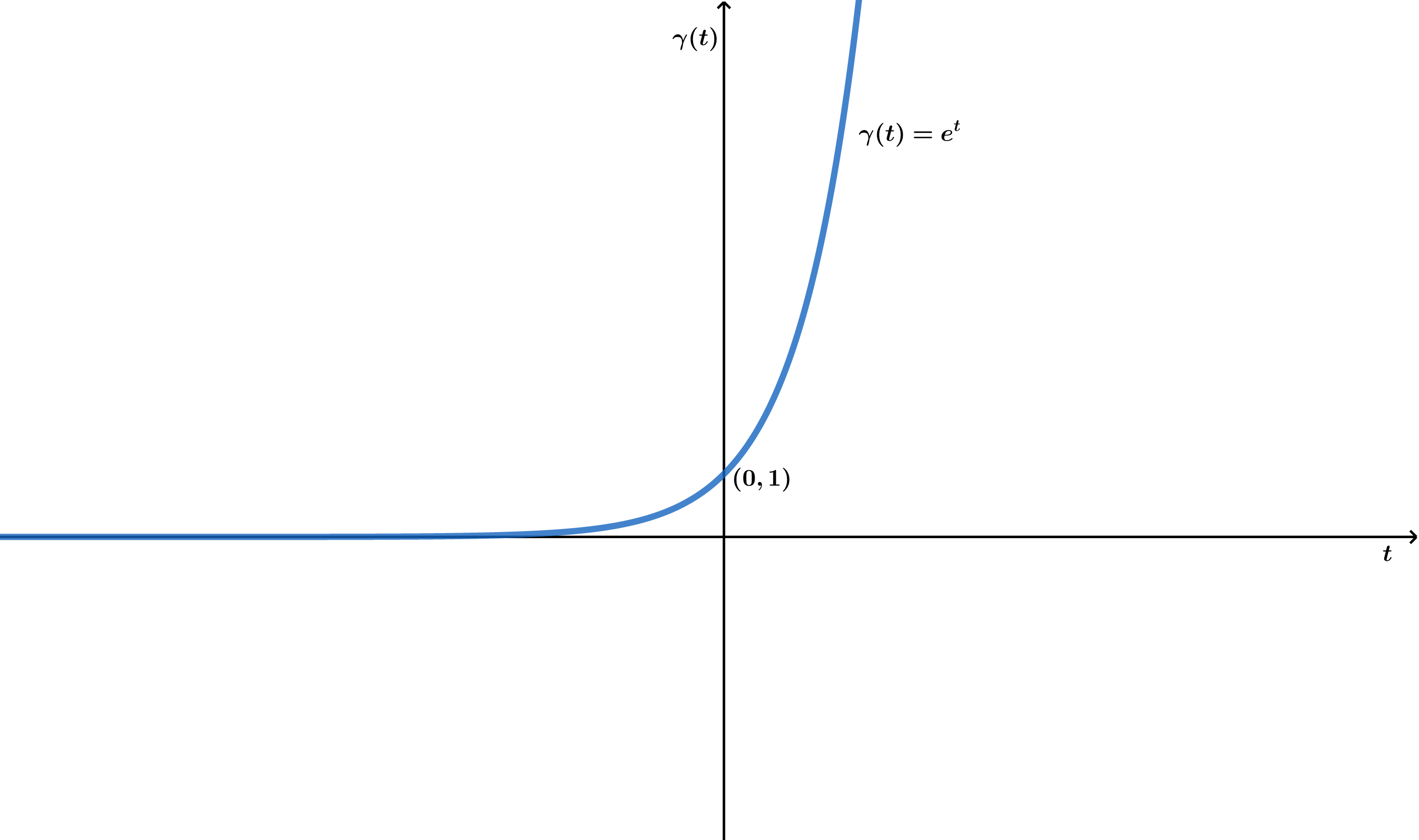}
	\caption{Exponential Generating Function}
	\label{fig2}
\end{figure}
By using the definition of $\gamma(.)$ given by \eqref{eq10h}, the constraint \eqref{eq10b} can be written as,
\begin{align}
\begin{aligned} \label{eq10i}
&  E_{{\bm{\zeta} \sim P} }\left [e^{\tau-\bm{\mu}(\bm{\zeta})^\top \bm x} \right ] \leq 1-\beta
\end{aligned}
\end{align}
Taking the logarithm in both sides of the constraint \eqref{eq10i} we get,
\begin{align}
\begin{aligned} \label{eq10j}
 & E_{{\bm{\zeta} \sim P} }\left [\tau-\bm{\mu}(\bm{\zeta})^\top \bm x \right ] \leq \log (1-\beta)
\end{aligned}
\end{align}
Now the left hand side of \eqref{eq10j} can be simplified in component form as,
\begin{align}
\begin{aligned} \label{eq11}
 E_{\bm{\zeta} \sim P} \left [{\tau-\bm{\mu}(\bm{\zeta})^\top \bm x} \right ] &=  E_{\bm{\zeta_j} \sim P_j} \left [{\tau-\sum_{j=1}^{n} \mu_j^{(0)} x_j-\sum_{j=1}^{n} (\mu_j^{(j)} x_j)\zeta_j} \right ]\\
& = \tau-\sum_{j=1}^{n} \mu_j^{(0)} x_j-\sum_{j=1}^{n} (\mu_j^{(j)} x_j)E_{\bm{\zeta_j} \sim P_j} \left [\zeta_j \right ]
\end{aligned}
\end{align}
Now our objective is to find out an upper bound of the right hand side of eq. \eqref{eq11}, which would be our $\psi^+(\bm{\mu}, \bm{x})$ given in Theorem \ref{prop1}. Since we have, $m_j^L \leq E_{{\zeta_j \sim P_j}}\left [\zeta_j \right ] \leq m_j^U$, it follows from eq. \eqref{eq11} that, 
\begin{align}
\begin{aligned} \label{eq11a}
E_{\bm{\zeta} \sim P} \left [{\tau-\bm{\mu}(\bm{\zeta})^\top \bm x} \right ] &\leq \tau-\sum_{j=1}^{n} \mu_j^{(0)} x_j-\sum_{j=1}^{n} (\mu_j^{(j)} x_j) m_j^L
\end{aligned}
\end{align} 
Then we get $\psi^+(\bm{\mu}, \bm{x})=\tau-\sum_{j=1}^{n} \mu_j^{(0)} x_j-\sum_{j=1}^{n} (\mu_j^{(j)} x_j) m_j^L$ and thus the safe approximation of the chance constraint becomes,
\begin{align*}
\tau-\sum_{j=1}^{n} \mu_j^{(0)} x_j-\sum_{j=1}^{n} (\mu_j^{(j)} x_j) m_j^L \leq \log(1-\beta)
\end{align*}
Therefore a safe approximation of the robust counterpart to the portfolio model \eqref{eq4} is given by,
\begin{align}
\begin{aligned} \label{eq12}
	&\min_{x_i} && \frac{1}{2} \sum_{i=1}^{n} \sum_{j=1}^{n} \sigma_{ij} x_i x_j \\
&\textrm{s.t.:} && \tau-\sum_{j=1}^{n} {\mu_j^{(0)} x_j}- \sum_{j=1}^n {(\mu_j^{(j)} x_j)m_j^L} \leq \log(1-\beta),\\
& && \sum_{j=1}^{n} {x_j}=1, \quad {x_j}\geq 0, \ j=1, 2, \dots, n.
\end{aligned}
\end{align}
The robust counterpart model \eqref{eq12} is a quadratic programming problem, so it is computationally tractable.\\ 
The above result can be represented as the following theorem:
\begin{theorem} \label{theorem2}
Consider the ambiguous chance constrained portfolio optimization model \eqref{eq4}, in which the probability distribution of the perturbations $\zeta_j$, $(j=1, 2, \dots, n)$ are not known precisely, and the only information regarding each $\zeta_j$ is that it belongs to a family of perturbations having known upper $\&$ lower mean bounds $m_j^U$ $\&$ $m_j^L$ respectively. Then a safe tractable convex approximation to this model can be a quadratic programming problem given by,
\begin{align*}
\begin{aligned}
&\min_{x_i} && \frac{1}{2} \sum_{i=1}^{n} \sum_{j=1}^{n} \sigma_{ij} x_i x_j \\
&\textrm{s.t.:} && \tau-\sum_{j=1}^{n} {\mu_j^{(0)} x_j}- \sum_{j=1}^n {(\mu_j^{(j)} x_j)m_j^L} \leq \log(1-\beta),\\
& && \sum_{j=1}^{n} {x_j}=1, \quad {x_j}\geq 0, \ j=1, 2, \dots, n.
\end{aligned}
\end{align*}
\end{theorem}
\subsection{Family of Perturbation Distributions with Known Upper $\&$ Lower Mean Bounds and Standard Deviations} \label{sec4.2}
In this case we consider the family of $\bm{\zeta}=[\zeta_1, \zeta_2, \dots, \zeta_n]^\top$ with the upper $\&$ lower bounds of each $\zeta_j$, $(j=1, 2, \dots, n)$ are known to be $m_j^U$ $\&$ $m_j^L$ respectively and the standard deviation is known to be $s_j$. For deriving the safe convex approximation we use a piecewise quadratic generating function.
\subsubsection{Piecewise Quadratic Generating Function Based Approximation}
The function
\begin{align}
\begin{aligned} \label{eq12a}
	\gamma(t)=\left \{ \begin{array}{ll}0 &\mbox{ for } t \leq -1\\
		(1+t)^2 &\mbox{ for } t > -1 \end{array} \right.
\end{aligned}
\end{align} 
is everywhere greater than or equal to $\chi(t)$. Further, it
is a non-negative and  non-decreasing convex function satisfying $\gamma(0) \geq 1$ and $\lim_{t \to -\infty}\gamma(t) =0$ and hence can be considered as a generating function.\\
\begin{figure}[htb]
	\centering
	\includegraphics[height=5cm,width=12cm]{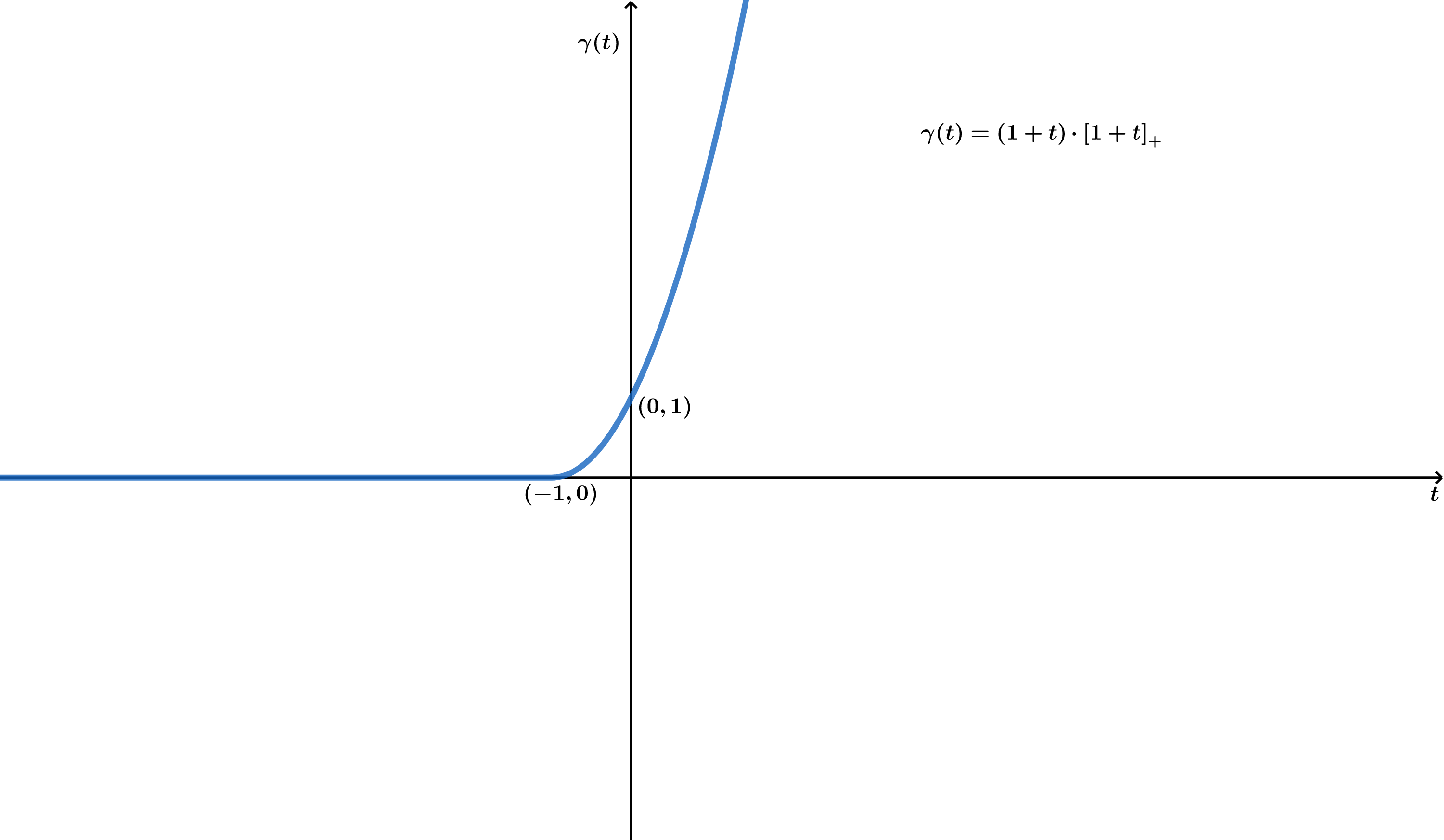}
	\caption{Piecewise Quadratic Generating Function}
	\label{fig3}
\end{figure}
We can also rewrite $\gamma(t)$ given in \eqref{eq12a} as,
\begin{align}
\begin{aligned} \label{eq12b}
&	\gamma(t)=(1+t) \cdot [1+t]_+=(1+t) \cdot \max \left \{0, 1+t \right \}
\end{aligned}
\end{align}
Similar to the previous two approximations, let's begin with the constraint
\begin{align}
\begin{aligned} \label{eq12c}
& E_{\bm{\zeta \sim P}} \left [ \gamma(\tau-\bm{\mu}(\bm{\zeta})^\top x) \right ] \leq 1-\beta
\end{aligned}
\end{align}
By the use of \eqref{eq12b}, the constraint \eqref{eq12c} becomes,
\begin{align}
\begin{aligned} \label{eq12d}
	 &  E_{\bm{\zeta \sim P}}  \left [(1+\tau-\bm{\mu}(\bm{\zeta})^\top x)\cdot \max \left \{0,(1+\tau-\bm{\mu}(\bm{\zeta})^\top x) \right \} \right ]  \leq 1-\beta
\end{aligned}
\end{align}
Now here arises two cases for $\max \left \{0, 1+\tau-\bm{\mu}(\bm{\zeta})^\top x \right \}$.\\
\textbf{Case I:} When $\max \left \{0, 1+\tau-\bm{\mu}(\bm{\zeta})^\top x \right \}=0$\\
Then the constraint \eqref{eq12d} becomes,
\begin{align*}
	\begin{aligned}
	 E_{\bm{\zeta \sim P}} \left [0 \right ]  \leq 1-\beta
	\end{aligned}
\end{align*}
\textbf{Case II:} When $\max \left \{0, 1+\tau-\bm{\mu}(\bm{\zeta})^\top x \right \}=1+\tau-\bm{\mu}(\bm{\zeta})^\top x$\\
Then the constraint \eqref{eq12d} becomes,
\begin{align*}
	\begin{aligned}
		E_{\bm{\zeta \sim P}} \left [1+\tau-\bm{\mu}(\bm{\zeta})^\top x \right ]^2  \leq 1-\beta
	\end{aligned}
\end{align*}
So the constraint \eqref{eq12d} can be replaced by two equivalent constraints,
\begin{subequations}
\begin{align}
	&  E_{\bm{\zeta \sim P}} \left [0 \right ] \leq 1-\beta \label{eq13i}\\
	& E_{\bm{\zeta \sim P}}  \left [1+\tau-\bm{\mu}(\bm{\zeta})^\top x \right ]^2  \leq 1-\beta \label{eq13ii}
\end{align}
\end{subequations}
Here the constraint \eqref{eq13i} is trivial, as by the assumption of chance constraint problems we have $1-\beta \geq 0$. Therefore now we only need to proceed with the constraint \eqref{eq13ii}.\\
Now the left hand side of \eqref{eq13ii} can be simplified in component form as,
\begin{align}
\begin{aligned} \label{eq17}
 E_{\bm{\zeta} \sim P} \left [1+\tau-\bm{\mu}(\bm{\zeta})^\top \bm x \right ]^2
	& = E_{\zeta_j \sim P_j} \left [1+\tau-\left (\sum_{j=1}^{n} \mu_j^{(0)} x_j+\sum_{j=1}^{n} (\mu_j^{(j)} x_j)\zeta_j\right ) \right ]^2\\
	& = E_{\zeta_j \sim P_j} \left [(1+\tau-\sum_{j=1}^{n} \mu_j^{(0)} x_j)-\sum_{j=1}^{n} (\mu_j^{(j)} x_j)\zeta_j \right ]^2\\
	& = E_{\zeta_j \sim P_j} \left [ \left (1+\tau-\sum_{j=1}^{n} \mu_j^{(0)} x_j \right )^2+ \left (\sum_{j=1}^{n} (\mu_j^{(j)} x_j)\zeta_j \right )^2 \right. \\
	& \qquad \qquad \qquad \qquad \qquad \left. -2\left (1+\tau-\sum_{j=1}^{n}\mu_j^{(0)}x_j\right ) \cdot \left (\sum_{j=1}^{n}(\mu_j^{(j)}x_j)\zeta_j \right )\right ]
\end{aligned}
\end{align}
Since in our problem the perturbations $\zeta_j$'s are the random variables, the terms without $\zeta_j$ can be treated as constants. So the eq. \eqref{eq17} can be written as,
\begin{align}
	\begin{aligned} \label{eq18}
	&	E_{\bm{\zeta} \sim P} \left [1+\tau-\bm{\mu}(\bm{\zeta})^\top \bm x \right ]^2= &&(1+\tau-\sum_{j=1}^{n} \mu_j^{(0)} x_j)^2+ E_{\zeta_j \sim P_j} \left [\left (\sum_{j=1}^{n} (\mu_j^{(j)} x_j)\zeta_j \right )\right ]^2\\
		& &&-2\left (1+\tau-\sum_{j=1}^{n}\mu_j^{(0)}x_j\right )  \cdot E_{\zeta_j \sim P_j} \left [ \sum_{j=1}^{n}(\mu_j^{(j)}x_j)\zeta_j\right ]
	\end{aligned}
\end{align}
Now our objective is to find out an upper bound of the right hand side of eq. \eqref{eq18}, which would be our $\psi^+(\bm{\mu}, \bm{x})$ given in Theorem \ref{prop1}. For this, we use the following lemma.
\begin{lemma} \label{lemma1}
Let $\zeta_j, \ j=1, 2, \dots, n$ be independently distributed random variables, each with its upper $\&$ lower mean bounds are respectively $m_j^U$ $\&$ $m_j^L$ and its standard deviation is $s_j$. Then we have,
\begin{align*}
\begin{aligned}
& (i) &&\sum_{j=1}^{n}(\mu_j^{(j)} x_j)m_j^L \leq E_{\zeta_j \sim P_j} \left [ \sum_{j=1}^{n}(\mu_j^{(j)}x_j)\zeta_j \right ]\leq \sum_{j=1}^{n}(\mu_j^{(j)} x_j)m_j^U,\\
& &&\text{ and, }\\
& (ii) &&\sum_{j=1}^{n}(\mu_j^{(j)}x_j)^2 s_j^2+ \left [\sum_{j=1}^{n}(\mu_j^{(j)} x_j)m_j^L\right ]^2 \leq E_{\zeta_j \sim P_j} \left [ \sum_{j=1}^{n}(\mu_j^{(j)}x_j)\zeta_j \right ]^2\leq \sum_{j=1}^{n}(\mu_j^{(j)}x_j)^2 s_j^2+ \left [\sum_{j=1}^{n}(\mu_j^{(j)} x_j)m_j^U\right ]^2
\end{aligned}
\end{align*}
\end{lemma}
\begin{proof} We first denote the random variable  $\sum_{j=1}^{n}(\mu_j^{(j)}x_j)\zeta_j$ by $Y$. So our objective is to find out the bounds of $E \left [Y \right ]$ and $E \left [Y^2 \right ]$.\\
Now the mean of $Y$ is given by,
\begin{align}
	\begin{aligned} \label{eq19}
		E \left [Y \right ] 
		& = \sum_{j=1}^{n}(\mu_j^{(j)}x_j) \cdot E_{\zeta_j \sim P_j} \left [ \zeta_j \right ]
	\end{aligned}
\end{align}
Using the fact that $m_j^L \leq E_{\zeta_j \sim P_j} \left [ \zeta_j \right ] \leq m_j^U$ in \eqref{eq19}, we get,
\begin{align}
\begin{aligned} \label{eq19a}
\sum_{j=1}^{n}(\mu_j^{(j)} x_j)m_j^L \leq E \left [Y \right ] \leq \sum_{j=1}^{n}(\mu_j^{(j)} x_j)m_j^U,
\end{aligned}
\end{align}
which proves (i).\\
Now the variance of $Y$ is given by,
\begin{align}
	\begin{aligned} \label{eq20}
		Var \left [Y \right ] 
		& = Var \left [\sum_{j=1}^{n}(\mu_j^{(j)}x_j)\zeta_j \right ] \\
		& = \sum_{j=1}^{n}(\mu_j^{(j)}x_j)^2 \cdot Var \left [\zeta_j \right ]+ 2 \cdot \displaystyle{\sum_{\substack{i,j \\ i\neq j}}} (\mu_i^{(i)}x_i) (\mu_j^{(j)}x_j) \cdot  Cov(\zeta_i, \zeta_j)
	\end{aligned}
\end{align}
Since $\zeta_j$'s are independently distributed random variables, we have $Cov(\zeta_i, \zeta_j)=0$, which reduces \eqref{eq20} to,
\begin{align}
\begin{aligned} \label{eq20a}
Var \left [Y \right ] 
& = \sum_{j=1}^{n}(\mu_j^{(j)}x_j)^2 s_j^2
\end{aligned}
\end{align}
Again from the variance formula we have,
\begin{align}
	\begin{aligned} \label{eq21}
	E\left [Y^2\right ]=Var\left [Y \right ]+\left (E\left [Y \right ]\right )^2
	\end{aligned}
\end{align}
Using \eqref{eq19a} and \eqref{eq20a} in \eqref{eq21}, we obtain
\begin{align}
\begin{aligned} \label{eq21a}
 & \sum_{j=1}^{n}(\mu_j^{(j)}x_j)^2 s_j^2+ \left[\sum_{j=1}^{n}(\mu_j^{(j)} x_j)m_j^L\right]^2 \leq E \left [ Y \right ]^2\leq \sum_{j=1}^{n}(\mu_j^{(j)}x_j)^2 s_j^2+ \left[\sum_{j=1}^{n}(\mu_j^{(j)} x_j)m_j^U\right]^2,
\end{aligned}
\end{align}
and this proves (ii).
\qed
\end{proof}
Now applying the results (i) and (ii) of Lemma \ref{lemma1} in eq. \eqref{eq18} we get,
\begin{align}
\begin{aligned} \label{eq22}
E_{\bm{\zeta} \sim P} \left [1+\tau-\bm{\mu}(\bm{\zeta})^\top \bm x \right ]^2 \leq &\left(1+\tau-\sum_{j=1}^{n} \mu_j^{(0)} x_j\right)^2+ \sum_{j=1}^{n}(\mu_j^{(j)} x_j)^2s_j^2+\left(\sum_{j=1}^{n}(\mu_j^{(j)} x_j)m_j^U\right)^2\\
& \qquad -2(1+\tau-\sum_{j=1}^{n}\mu_j^{(0)}x_j)  \cdot \left (\sum_{j=1}^{n}(\mu_j^{(j)} x_j)m_j^L\right )
\end{aligned}
\end{align}
Then we get $\psi^+(\bm{\mu}, \bm{x})=\left(1+\tau-\sum_{j=1}^{n} \mu_j^{(0)} x_j\right)^2+ \sum_{j=1}^{n}(\mu_j^{(j)} x_j)^2s_j^2+\left (\sum_{j=1}^{n}(\mu_j^{(j)} x_j)m_j^U\right )^2-2(1+\tau-\sum_{j=1}^{n}\mu_j^{(0)}x_j)  \cdot \left (\sum_{j=1}^{n}(\mu_j^{(j)} x_j)m_j^L\right )$ and thus the safe approximation of the chance constraint becomes,
\begin{align*}
\begin{aligned}	
	& \left(1+\tau-\sum_{j=1}^{n} \mu_j^{(0)} x_j\right)^2+ \sum_{j=1}^{n}(\mu_j^{(j)} x_j)^2s_j^2+\left (\sum_{j=1}^{n}(\mu_j^{(j)} x_j)m_j^U\right )^2-2(1+\tau-\sum_{j=1}^{n}\mu_j^{(0)}x_j)  \cdot \left (\sum_{j=1}^{n}(\mu_j^{(j)} x_j)m_j^L\right )\\
	& \qquad \qquad \qquad \qquad \qquad \qquad \qquad \qquad \qquad \qquad \qquad \qquad \qquad \qquad \qquad \qquad \qquad \qquad \qquad \leq 1-\beta
\end{aligned}
\end{align*}
Therefore a safe approximation of the robust counterpart to the portfolio model \eqref{eq4} is given by,
\begin{align}
\begin{aligned}	\label{eq22a}
	&\min_{x_i} && \frac{1}{2} \sum_{i=1}^{n} \sum_{j=1}^{n} \sigma_{ij} x_i x_j \\
	&\textrm{s.t.:} &&\left(1+\tau-\sum_{j=1}^{n} \mu_j^{(0)} x_j\right)^2+ \sum_{j=1}^{n}(\mu_j^{(j)} x_j)^2s_j^2+\left (\sum_{j=1}^{n}(\mu_j^{(j)} x_j)m_j^U\right )^2\\
	& && \qquad \qquad \qquad -2(1+\tau-\sum_{j=1}^{n}\mu_j^{(0)}x_j)  \cdot \left (\sum_{j=1}^{n}(\mu_j^{(j)} x_j)m_j^L\right ) \leq 1-\beta,\\
	& &&\sum_{j=1}^{n} {x_j}=1, \quad {x_j}\geq 0, \ j=1, 2, \dots, n.
\end{aligned}
\end{align}
The robust counterpart model \eqref{eq22a} is a quadratically constrained quadratic programming (QCQP) problem, so it is computationally tractable.\\
The above result can be represented as the following theorem:
\begin{theorem} \label{theorem3}
Consider the ambiguous chance constrained portfolio optimization model \eqref{eq4}, in which the probability distribution of the perturbations $\zeta_j$, $(j=1, 2, \dots, n)$ are not known precisely, and the only information regarding each $\zeta_j$ is that it belongs to a family of perturbations having known upper $\&$ lower mean bounds, and the standard deviation $m_j^U$ $\&$ $m_j^L$, and $s_j$ respectively. Then a safe tractable convex approximation to this model can be a quadratic programming problem given by,
\begin{align*}
\begin{aligned}	
&\min_{x_i} && \frac{1}{2} \sum_{i=1}^{n} \sum_{j=1}^{n} \sigma_{ij} x_i x_j \\
&\textrm{s.t.:} &&\left(1+\tau-\sum_{j=1}^{n} \mu_j^{(0)} x_j\right)^2+ \sum_{j=1}^{n}(\mu_j^{(j)} x_j)^2s_j^2+\left (\sum_{j=1}^{n}(\mu_j^{(j)} x_j)m_j^U\right )^2\\
& && \qquad \qquad \qquad -2(1+\tau-\sum_{j=1}^{n}\mu_j^{(0)}x_j)  \cdot \left (\sum_{j=1}^{n}(\mu_j^{(j)} x_j)m_j^L\right ) \leq 1-\beta,\\
& &&\sum_{j=1}^{n} {x_j}=1, \quad {x_j}\geq 0, \ j=1, 2, \dots, n.
\end{aligned}
\end{align*}
\end{theorem}
\subsection{Extension to Higher Orders}
In Section \ref{sec4.1}, the considered perturbation family has the information regarding first-moment statistics of the distribution, that is, the mean. On the other hand, the perturbation family in Section \ref{sec4.2} has information regarding the first two moments of the distribution, that is, the mean and the variance (or the standard deviation). A piecewise linear function \eqref{eq10a} is used to derive the safe convex approximation for the former family, whereas a piecewise quadratic function \eqref{eq12a} is used for the latter family.\\ Inductively, if we have the information regarding the first $n$ moments of the perturbation distribution, then we can use the piecewise generating function of degree $n$ given by,
\begin{align}
\begin{aligned} \label{eq22aa}
\gamma(t)=\left \{ \begin{array}{ll}0 &\mbox{ for } t \leq -1\\
(1+t)^n &\mbox{ for } t > -1 \end{array} \right.
\end{aligned}
\end{align}
to derive the safe convex approximation.\\
This is because, while deriving the left hand side of the chance constraint componentwise by using the above generating function, we obtain
\begin{align}
\begin{aligned} \label{eq22aaa}
&	E_{\bm{\zeta} \sim P} \left [1+\tau-\bm{\mu}(\bm{\zeta})^\top \bm x \right ]^n= && E_{\zeta_j \sim P_j} \left [(1+\tau-\sum_{j=1}^{n} \mu_j^{(0)} x_j)-\sum_{j=1}^{n} (\mu_j^{(j)} x_j)\zeta_j \right ]^n
\end{aligned}
\end{align}
After applying the binomial expansion to the right hand side of eq. \eqref{eq22aaa}, we require the information regarding the values of $E_{\zeta_j \sim P_j} [\zeta_j]$, $E_{\zeta_j \sim P_j} [\zeta_j^2]$, $\dots$, $E_{\zeta_j \sim P_j} [\zeta_j^n]$ to derive the safe convex approximation of the uncertain model. Therefore the generating function \eqref{eq22aa} can be applied for the family of perturbations, whose first $n$ moments' information is known.
\section{Numerical Example} \label{section5}
In this section, we consider an Indian stock market problem for our study. The objective is to find out the optimal allocation by using the safe approximation models which we derive in the last section. Finally, we compare the obtained results. 
\subsection{Data Description}
The stock price data of three major sectors of India such as Nifty Bank, Nifty Infra, and Nifty IT are taken into consideration for our numerical example. The data are collected for the period June 2017 to May 2022 from \href{https://finance.yahoo.com}{https://finance.yahoo.com}. The quarterly returns of the sectors are derived from their stock price data. Using these we calculate the input parameters for the portfolio model namely, the expected returns of the sectors and the covariance of returns between each pair of sectors. The calculated input parameters are the nominal values, which are given in Table \ref{table1}.\\
\begin{table}[htb]
	\centering
	\caption{Nominal Values of the Input Parameters}
	\begin{tabular}{| c | c | c  c  c |}
		\hline
		\cline{1-5}
		\textbf{Sectors} &  {\textbf{Expected Returns (in $\%$)}} & \multicolumn{3}{c |} {\textbf{Covariance of Returns (in $\%$)}} \\
		\cline{3-5}
		{} & {} & Nifty Bank & Nifty Infra & Nifty IT  \\
		\hline
		Nifty Bank & $2.609$ & $24.126 $ & $-1.460 $ & $11.032$  \\
		
		Nifty Infra & $-1.430$ & $-1.460 $ & $8.237$ & $0.461$   \\
		
		Nifty IT & $6.329 $ & $11.032$ & $0.461$ & $18.034$   \\
		
		\hline
	\end{tabular}
	\label{table1}
\end{table}
\subsection{Problem Formulation}
First, we construct the nominal problem with the input parameters given in Table \ref{table1}. The problem is given by,
\begin{align} 
\begin{aligned} \label{eq22b}	
&\min_{x_i} && \frac{1}{2} \left[24.126x_1^2+8.237x_2^2+18.034x_3^2+ 2 \cdot (-1.460) x_1x_2+2 \cdot 11.032 x_1x_3+ 2 \cdot 0.461 x_2x_3 \right]\\
& s.t.: && 2.609x_1-1.430x_2+6.329x_3 \geq \tau,\\
& &&x_1+x_2+x_3=1, \quad x_1, x_2, x_3 \geq 0 
\end{aligned}
\end{align}
Taking the uncertainty in expected returns into account, we aim to formulate the problem with ambiguous chance constraint. Let $\bm{\zeta}=[\zeta_1 \ \zeta_2 \  \zeta_3]^\top$ be the perturbation vector associated with the expected returns.\\
Here the nominal expected return vector is given by,
\begin{align*}
	\bm{{\bm{\mu}^{(0)}}} = \begin{bmatrix} 
		\mu_1^{(0)} \\
		\mu_2^{(0)}\\
		\mu_3^{(0)} \\
	\end{bmatrix}
	= \begin{bmatrix} 
		2.609 \\
		-1.430 \\
		6.329 \\
	\end{bmatrix}
\end{align*}
Let the basic shifts of the perturbations are assumed to be,
\begin{align*}
	\bm{{\bm{\mu}^{(1)}}} = \begin{bmatrix} 
		\mu_1^{(1)} \\
		\mu_2^{(1)}\\
		\mu_3^{(1)} \\
	\end{bmatrix}
	= \begin{bmatrix} 
		0.2 \\
		0 \\
		0 \\
	\end{bmatrix}, 
	\quad 
	\bm{{\bm{\mu}^{(2)}}} = \begin{bmatrix} 
		\mu_1^{(2)} \\
		\mu_2^{(2)}\\
		\mu_3^{(2)} \\
	\end{bmatrix}
	= \begin{bmatrix} 
		0 \\
		0.1 \\
		0 \\
	\end{bmatrix},
	\quad
	\bm{{\bm{\mu}^{(3)}}} = \begin{bmatrix} 
		\mu_1^{(3)} \\
		\mu_2^{(3)}\\
		\mu_3^{(3)} \\
	\end{bmatrix}
	= \begin{bmatrix} 
		0 \\
		0 \\
		0.3 \\
	\end{bmatrix}.
\end{align*}
Then the uncertain expected return of the portfolio is given by,
\begin{align*}
\begin{aligned}
&[\bm{\mu}(\bm{\zeta})]^\top \bm{x}&&= \left [ \bm{\mu}^{(0)}+\sum_{j=1}^{3}\zeta_j \bm{\mu}^{(j)} \right ]^\top \bm x\\
& &&= \begin{bmatrix} 
2.609+0.2\zeta_1 \\
-1.430+0.1\zeta_2 \\
6.329+0.3\zeta_3 \\
\end{bmatrix}^\top \begin{bmatrix} 
x_1 \\
x_2 \\
x_3 \\
\end{bmatrix}\\
& &&= 2.609x_1-1.430x_2+6.329x_3+0.2x_1 \zeta_1+0.1 x_2 \zeta_2+ 0.3x_3 \zeta_3
\end{aligned}
\end{align*}
Therefore the ambiguous chance constrained problem for $\beta=0.95$ is given by,
\begin{align} 
\begin{aligned}	\label{eq23}
&\min_{x_i} && \frac{1}{2} \left[24.126x_1^2+8.237x_2^2+18.034x_3^2+ 2 \cdot (-1.460) x_1x_2+2 \cdot 11.032 x_1x_3+ 2 \cdot 0.461 x_2x_3 \right]\\
& s.t.: && \forall P \in \mathcal{P}: {Prob}_{\bm{\zeta} \sim P} \left \{2.609x_1-1.430x_2+6.329x_3+0.2x_1 \zeta_1+0.1 x_2 \zeta_2+ 0.3x_3 \zeta_3\geq \tau \right \} \geq 0.95\\
&  && x_1+x_2+x_3=1, \quad x_1, x_2, x_3 \geq 0 
\end{aligned}
\end{align}
We aim to solve the above ambiguous chance constrained problem for the families of perturbation distributions by using the safe approximations discussed in Section \ref{section4}.
\subsection{Optimal Solutions for Several Families of Perturbation Distributions}
\subsubsection{Family of Perturbation Distributions with Known Upper $\&$ Lower Mean Bounds}
Consider the family $\mathcal{P}$ of perturbation distributions, whose upper $\&$ lower mean bounds are respectively known as,
\begin{align}
\begin{aligned} \label{eq23a} 
\bm{{m}}^U  = \begin{bmatrix} 
m_1^{U} \\
m_2^{U}\\
m_3^{U}\\
\end{bmatrix}
= \begin{bmatrix} 
0.3 \\
0.2 \\
0.1\\
\end{bmatrix}, \quad 
\bm{{m}}^L = \begin{bmatrix} 
m_1^{L} \\
m_2^{L}\\
m_3^{L} \\
\end{bmatrix}
= \begin{bmatrix} 
-0.3 \\
-0.2 \\
-0.1 \\
\end{bmatrix}
\end{aligned}
\end{align}
\textbf{a. Piecewise Linear Generating Function Based Approximation}\\
Using Theorem \ref{theorem1} we get the safe convex approximation with the piecewise linear generating function as,
\begin{align} 
\begin{aligned}	\label{eq24}
&\min_{x_i} && \frac{1}{2} \left[24.126x_1^2+8.237x_2^2+18.034x_3^2+ 2 \cdot (-1.460) x_1x_2+2 \cdot 11.032 x_1x_3+ 2 \cdot 0.461 x_2x_3 \right]\\
& s.t.: && 1+\tau- (2.609x_1-1.430x_2+6.329x_3)-(0.2x_1m_1^L+0.1x_2m_2^L+0.3x_3m_3^L) \leq 0.05,\\
& && x_1+x_2+x_3=1, \quad x_1, x_2, x_3 \geq 0,
\end{aligned}
\end{align}
where the values of $m_1^L$, $m_2^L$, $m_3^L$ are given in eq. \eqref{eq23a}.\\
 We solve this problem for different values of $\tau$ varying from $1.5$ to $3.5$ and calculate the optimal portfolio risks. Such choice of $\tau$ is because of the fact that the portfolio is a combination of all the individual assets and thus there is a high possibility of the portfolio return to lie somewhere near the average of the asset returns. The obtained optimal results are provided in Table \ref{table2}.\\
\begin{table}[htb]
	\centering
	\caption{Optimal Results for Piecewise Linear Generating Function Based Approximation}
	\begin{tabular}{| c | c  c  c | c |}
		\hline
		\cline{1-5}
		\textbf{Target Return($\tau$)} &  \multicolumn{3}{c |} {\textbf{Optimal Allocation}} &  {\textbf{Optimal Portfolio Risk}} \\
		\cline{2-4}
		{}  & Nifty Bank & Nifty Infra & Nifty IT & {} \\
		\hline
		$1.5$ & $0.0979$ & $0.4493$ & $0.4528 $ & $3.3142$  \\
		
		$1.7$ & $0.0891$ & $0.4278$ & $0.4831$ & $3.4685$   \\
		
		$1.9$ & $0.0803$ & $0.4062$ & $0.5134$ & $3.6382$   \\
		
		$2.1$ & $0.0716$ & $0.3847$ & $0.5438 $ & $3.8231$  \\
		
		$2.3$ & $0.0628$ & $0.3631$ & $0.5741 $ & $4.0232$   \\
		
		$2.5$ & $0.0540$ & $0.3415$ & $0.6045$ & $4.2386$   \\
		
		$2.7$ & $0.0452$ & $0.3200$ & $0.6348$ & $4.4693$  \\
		
		$2.9$ & $0.0364$ & $0.2984$ & $0.6652 $ & $4.7152$  \\
		
		$3.1$ & $0.0276$ & $0.2769$ & $0.6955$ & $4.9763$   \\
		
		$3.3$ & $0.0189$ & $0.2553$ & $0.7259$ & $5.2528$  \\
		
		$3.5$ & $0.0101$ & $0.2337$ & $0.7562$ & $5.5444$   \\
		
		\hline
	\end{tabular}
	\label{table2}
\end{table}
Using these results, we plot the optimal portfolio allocation graph in Fig. \ref{fig4}, where the black, red, and blue respectively indicate the weights of Nifty Bank, Nifty Infra, and Nifty IT in the optimal portfolio, when the target return varies from 1.5 to 3.5. Moreover, the tradeoff between target returns and their corresponding optimal risks is illustrated by the efficient frontier, which is given in Fig. \ref{fig5}. 
\begin{figure}[htb] 
\centering 
	\includegraphics[height=8 cm]{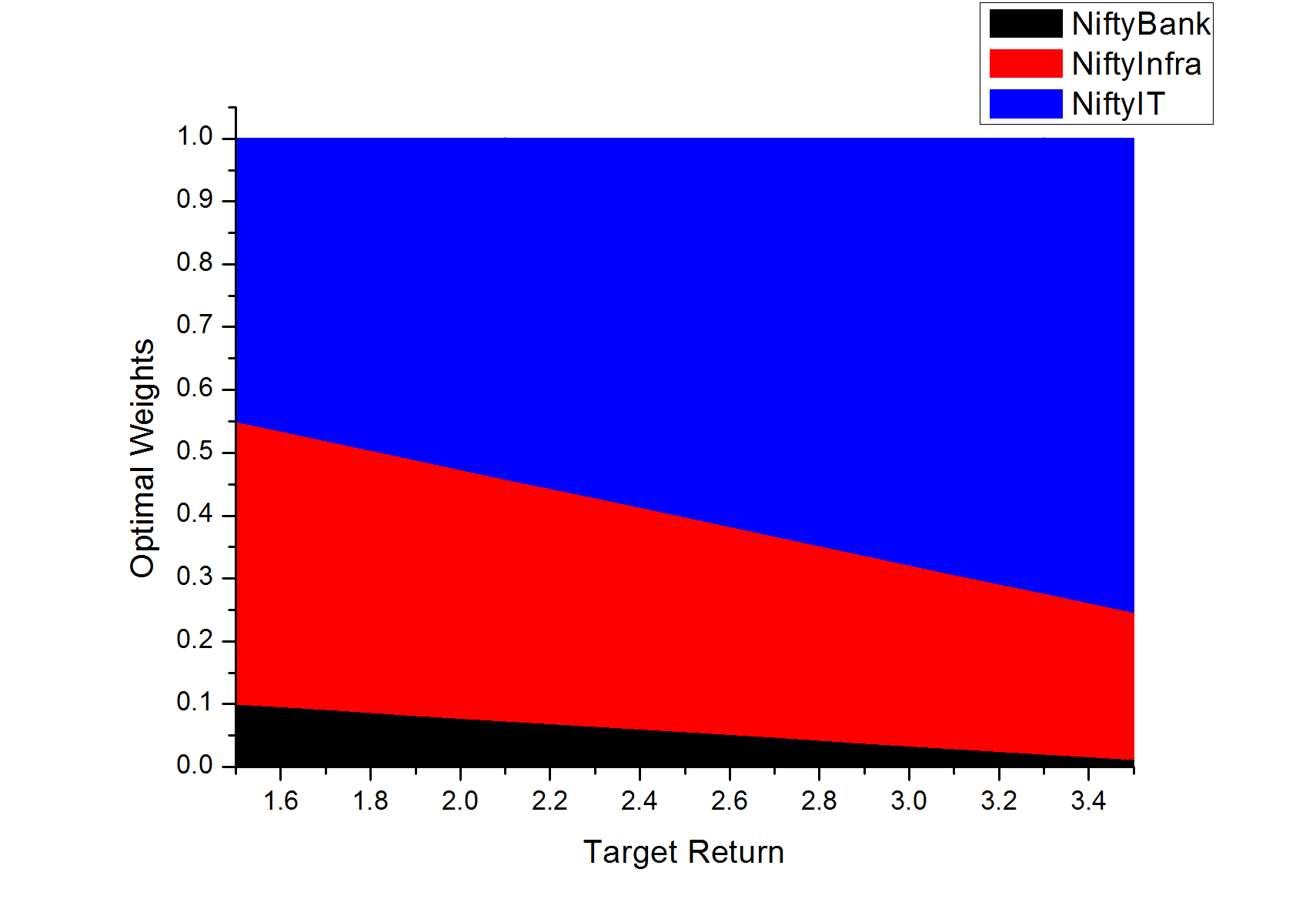}
	\caption{Portfolio Allocation for the Piecewise Linear Generating Function Based Approximation}
\label{fig4}
\end{figure}
\begin{figure}[htb]
\centering 
	\includegraphics[height=8 cm]{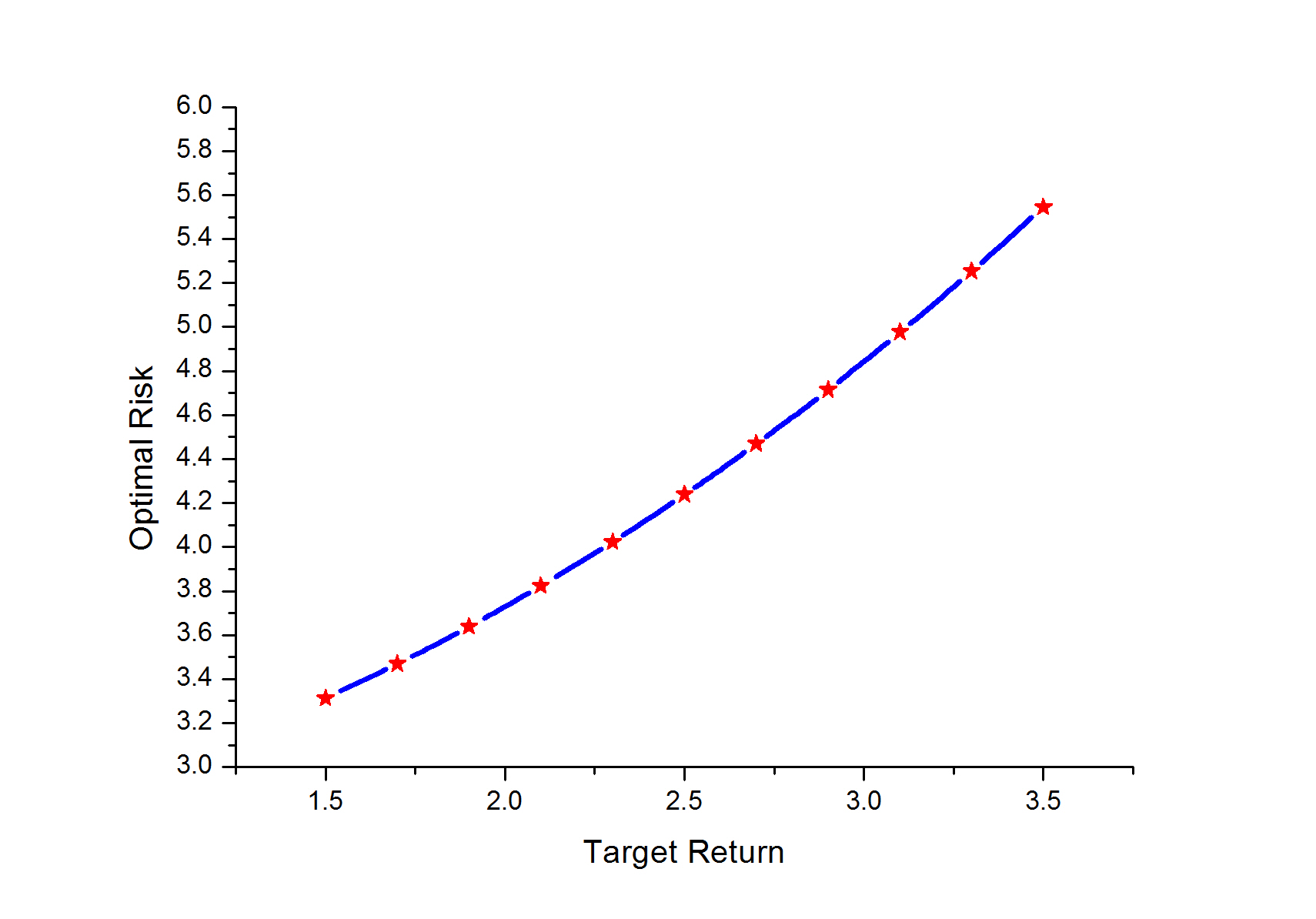}
	\caption{Efficient Frontier of the Portfolio for the Piecewise Linear Generating Function Based Approximation}
\label{fig5}
\end{figure}\\
\textbf{b. Bernstein Approximation}\\
Using Theorem \ref{theorem2} the safe convex approximation for Bernstein approximation is given by,
\begin{align} 
\begin{aligned}	\label{eq25}
	&\min_{x_i} && \frac{1}{2} \left[24.126x_1^2+8.237x_2^2+18.034x_3^2+ 2 \cdot (-1.460) x_1x_2+2 \cdot 11.032 x_1x_3+ 2 \cdot 0.461 x_2x_3 \right]\\
	& s.t.: && \tau- (2.609x_1-1.430x_2+6.329x_3)-(0.2x_1m_1^L+0.1x_2m_2^L+0.3x_3m_3^L) \leq \log(0.05),\\
	& && x_1+x_2+x_3=1, \quad x_1, x_2, x_3 \geq 0,
\end{aligned}
\end{align}
where the values of $m_1^L$, $m_2^L$, $m_3^L$ are given in eq. \eqref{eq23a}.\\
 We solve this problem for the values of $\tau$ varying from $1.5$ to $3.5$ and calculate the optimal portfolio risks. The obtained optimal results are provided in Table \ref{table3}.\\
\begin{table}[h!]
	\centering
	\caption{Optimal Results for the Bernstein Approximation}
	\begin{tabular}{| c | c  c  c | c |}
		\hline
		\cline{1-5}
		\textbf{Target Return($\tau$)} &  \multicolumn{3}{c |} {\textbf{Optimal Allocation}} &  {\textbf{Optimal Portfolio Risk}} \\
		\cline{2-4}
		{}  & Nifty Bank & Nifty Infra & Nifty IT & {} \\
		\hline
		$1.5$ & $0.0081$ & $0.2288$ & $0.7631 $ & $5.6133$  \\
		
		$1.7$ & $0.0000$ & $0.2069$ & $0.7931$ & $5.9237$   \\
		
		$1.9$ & $0.0000$ & $0.1811$ & $0.8189$ & $6.2503$   \\
		
		$2.1$ & $0.0000$ & $0.1553$ & $0.8447 $ & $6.5939$  \\
		
		$2.3$ & $0.0000$ & $0.1295$ & $0.8705 $ & $6.9543$   \\
		
		$2.5$ & $0.0000$ & $0.1037$ & $0.8963$ & $7.3316$   \\
		
		$2.7$ & $0.0000$ & $0.0779$ & $0.9221$ & $7.7257$  \\
		
		$2.9$ & $0.0000$ & $0.0520$ & $0.9480 $ & $8.1368$  \\
		
		$3.1$ & $0.0000$ & $0.0262$ & $0.9738$ & $8.5648$   \\
		
		$3.3$ & $0.0000$ & $0.0004$ & $0.9996$ & $9.0096$  \\
		
		$3.5$ & $0.0000$ & $0.0000$ & $1.0000$ & $9.5839$   \\
		
		\hline
	\end{tabular}
	\label{table3}
\end{table}
Using these results, we plot the optimal allocation graph and the efficient frontier, which are illustrated in Fig. \ref{fig6} and Fig. \ref{fig7} respectively.
\begin{figure}[htb]
\centering 
	\includegraphics[height=8 cm]{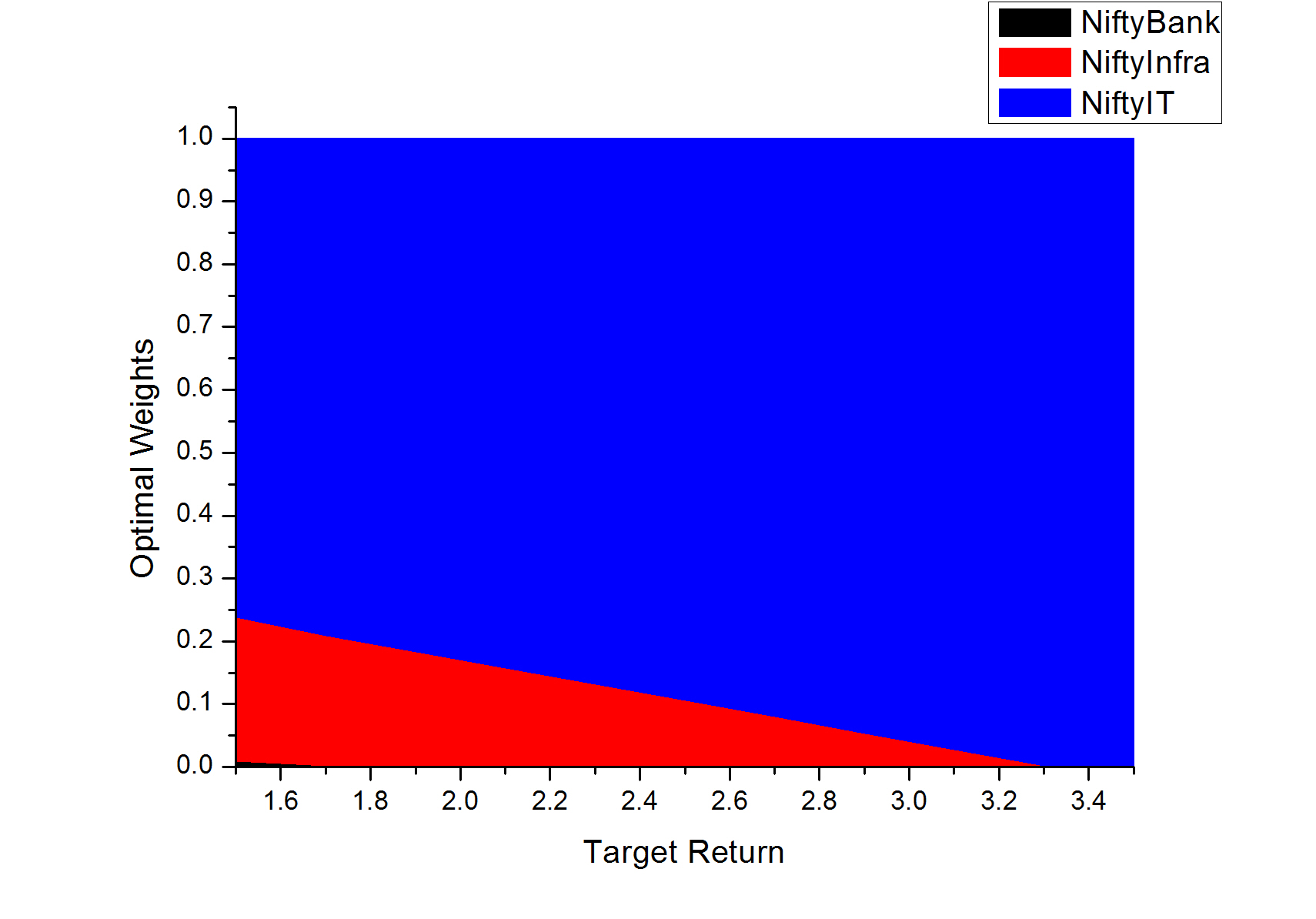}
	\caption{Portfolio Allocation for the Bernstein Approximation}
\label{fig6}
\end{figure}
\begin{figure}[htb]
\centering 
	\includegraphics[height=8 cm]{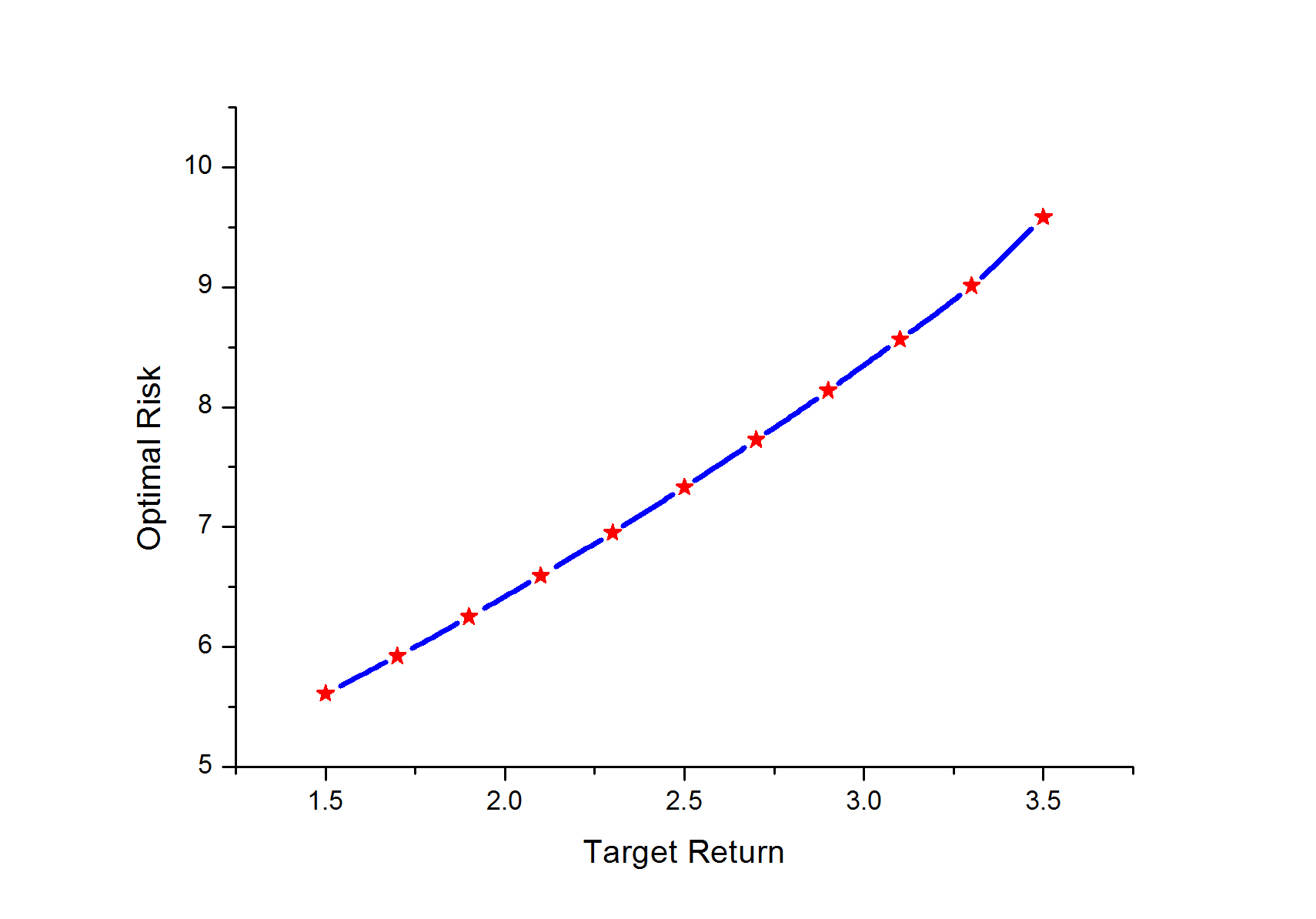}
	\caption{Efficient Frontier of the Portfolio for the Bernstein Approximation}
\label{fig7}
\end{figure}
\subsubsection{Family of Perturbation Distributions with Known Upper $\&$ Lower Mean Bounds and Standard Deviations}
In this case, we consider the family $\mathcal{P}$ of perturbation distributions, whose upper $\&$ lower mean bounds are given by eq. \eqref{eq23a} and the standard deviations are given by,
\begin{align}
\begin{aligned} \label{eq25a}
\bm{s} = \begin{bmatrix} 
s_1 \\
s_2\\
s_3 \\
\end{bmatrix}
= \begin{bmatrix} 
0.1 \\
0.1 \\
0.1 \\
\end{bmatrix}
\end{aligned}
\end{align}
\textbf{a. Piecewise Quadratic Approximation}\\
Using Theorem \ref{theorem3} we get the safe convex approximation with the piecewise quadratic generating function as,
\begin{align} 
\begin{aligned}	\label{eq26}
&\min_{x_i} && \frac{1}{2} \left[24.126x_1^2+8.237x_2^2+18.034x_3^2+ 2 \cdot (-1.460) x_1x_2+2 \cdot 11.032 x_1x_3+ 2 \cdot 0.461 x_2x_3 \right]\\
& s.t.: && [(1+\tau-(2.609x_1-1.430x_2+6.329x_3)]^2+[(0.2x_1)^2s_1^2+(0.1x_2)^2s_2^2+(0.3x_3)^2s_3^2]\\
& &&+[(0.2x_1)m_1^U+(0.1x_2)m_2^U+(0.3x_3)m_3^U]^2\\
& && -2\left [1+\tau- (2.609x_1-1.430x_2+6.329x_3)\right ]\cdot \left [0.2x_1(m_1^L)+0.1x_2(m_2^L)+0.3x_3(m_3^L)\right ] \leq 0.05,\\
& && x_1+x_2+x_3=1, \quad x_1, x_2, x_3 \geq 0,
\end{aligned}
\end{align}
where the values of $m_j^L$ $\&$ $m_j^U$, ($j=1,2,3$) are given in eq. \eqref{eq23a}, and the values of $s_j$, ($j=1,2,3$) are given in eq. \eqref{eq25a}.\\
We solve this problem for the values of $\tau$ varying from $1.5$ to $3.5$ and calculate the optimal portfolio risks. The obtained optimal results are provided in Table \ref{table4}.\\
\begin{table}[htb]
	\centering
	\caption{Optimal Results for Piecewise Quadratic Function Based Approximation}
	\begin{tabular}{| c | c  c  c | c |}
		\hline
		\cline{1-5}
		\textbf{Target Return($\tau$)} &  \multicolumn{3}{c |} {\textbf{Optimal Allocation}} &  {\textbf{Optimal Portfolio Risk}} \\
		\cline{2-4}
		{}  & Nifty Bank & Nifty Infra & Nifty IT & {} \\
		\hline
		$1.5$ & $0.1060$ & $0.4584$ & $0.4356 $ & $3.2423$  \\
		
		$1.7$ & $0.0974$ & $0.4362$ & $0.4664$ & $3.3924$   \\
		
		$1.9$ & $0.0888$ & $0.4140$ & $0.4972$ & $3.5585$   \\
		
		$2.1$ & $0.0803$ & $0.3917$ & $0.5280 $ & $3.7405$  \\
		
		$2.3$ & $0.0717$ & $0.3695$ & $0.5588 $ & $3.9384$   \\
		
		$2.5$ & $0.0631$ & $0.3473$ & $0.5896$ & $4.1523$   \\
		
		$2.7$ & $0.0546$ & $0.3250$ & $0.6204$ & $4.3821$  \\
		
		$2.9$ & $0.0460$ & $0.3028$ & $0.6512 $ & $4.6279$  \\
		
		$3.1$ & $0.0374$ & $0.2806$ & $0.6820$ & $4.8896$   \\
		
		$3.3$ & $0.0289$ & $0.2583$ & $0.7128$ & $5.1673$  \\
		
		$3.5$ & $0.0203$ & $0.2361$ & $0.7436$ & $5.4609$   \\
		
		\hline
	\end{tabular}
	\label{table4}
\end{table}
Using these results, we plot the optimal portfolio allocation graph and efficient frontier, which are illustrated in Fig. \ref{fig8} and Fig. \ref{fig9} respectively.
\begin{figure}[htb]
\centering
\includegraphics[height=8 cm]{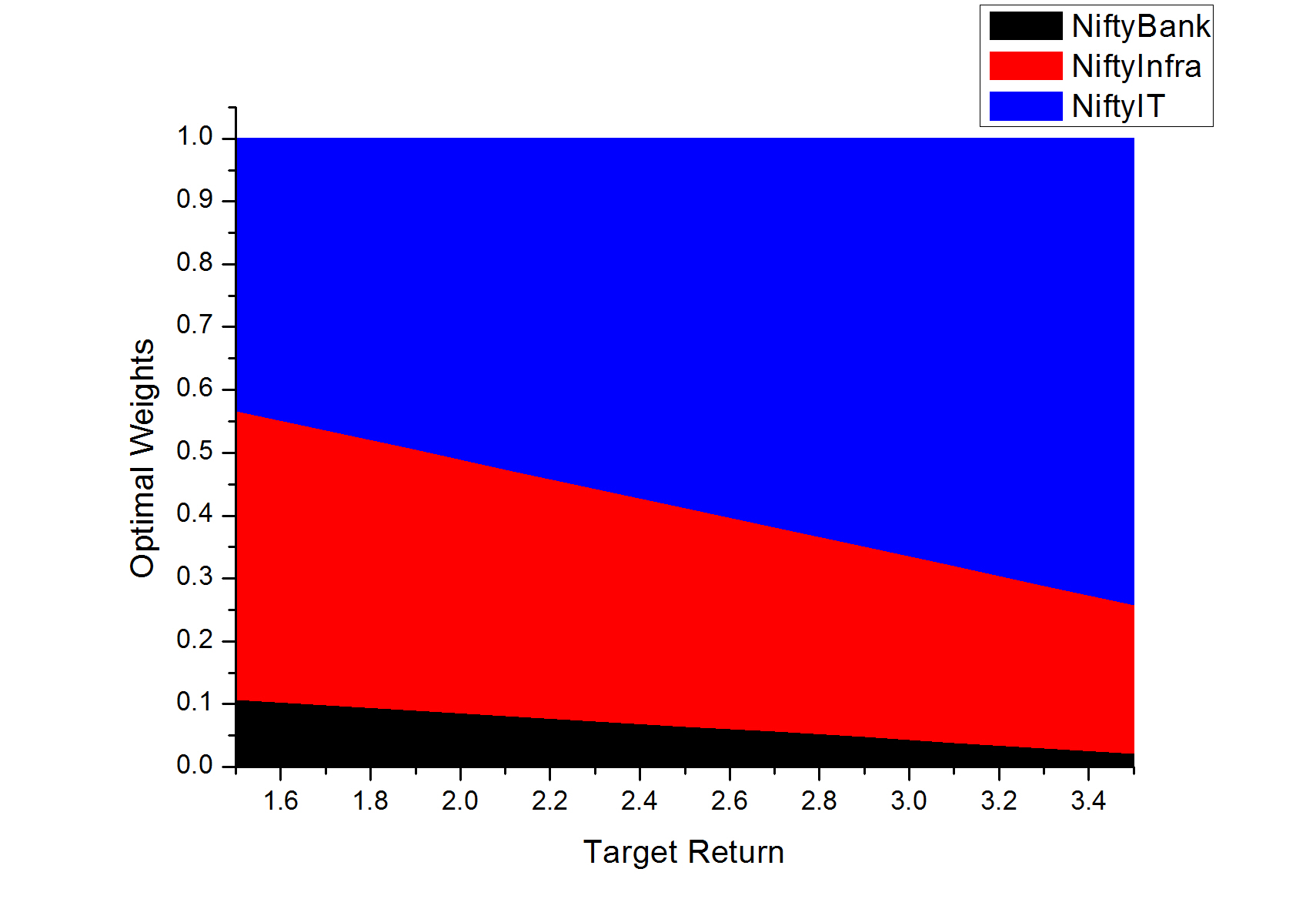}
\caption{Portfolio Allocation for the Piecewise Quadratic Function Based Approximation}
\label{fig8}
\end{figure}
\begin{figure}[htb]
\centering 
\includegraphics[height=8 cm]{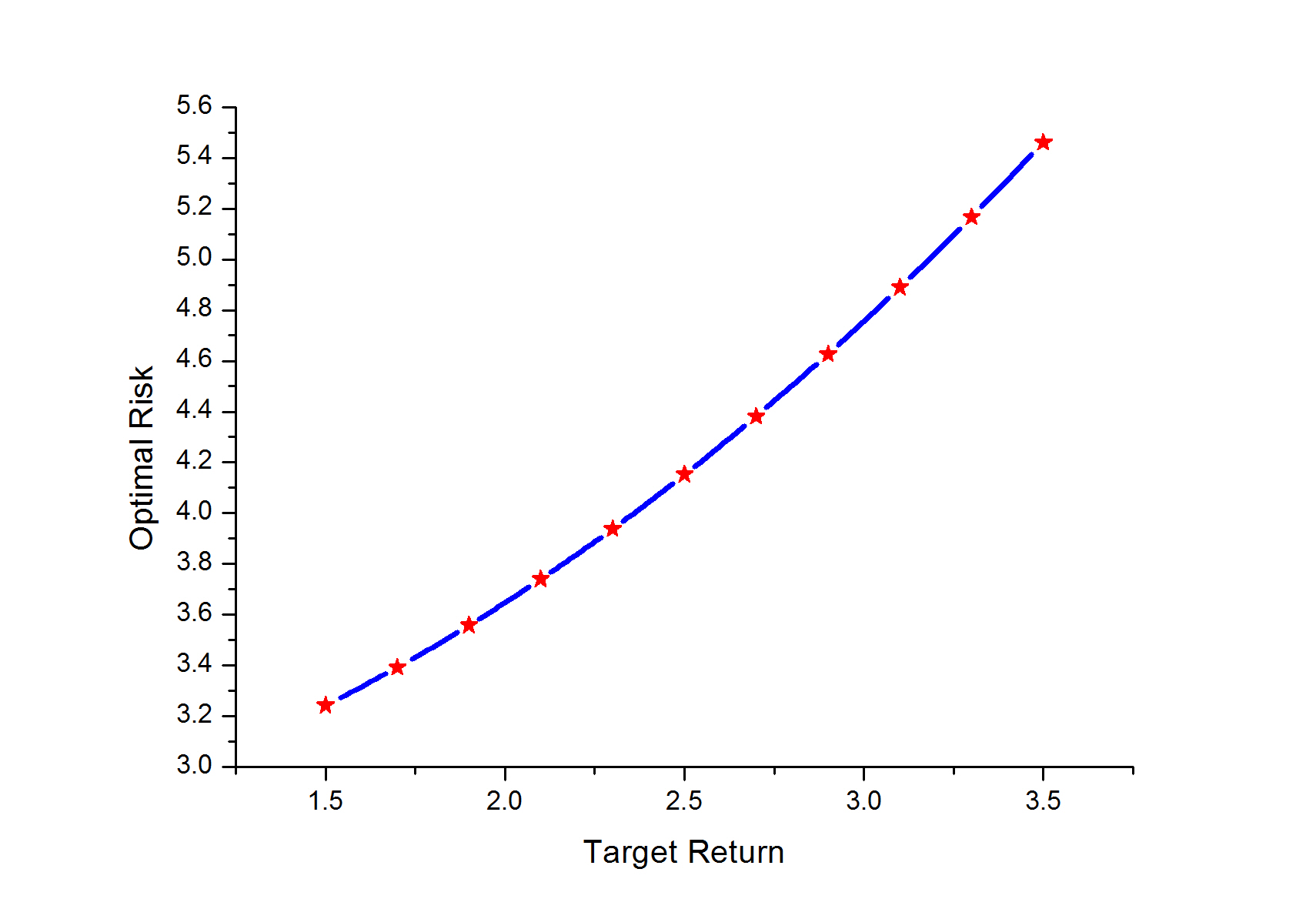}
\caption{Efficient Frontier of the Portfolio for the Piecewise Quadratic Generating Function Based Approximation}
\label{fig9}
\end{figure}
\section{Conclusion} \label{section6}
In this paper, we derive safe convex approximations for the ambiguous chance constraint based uncertain portfolio optimization problems.  Our study comprises two families of perturbation distributions-- (i) family with known upper $\&$ lower mean bounds, and (ii) family with known upper $\&$ lower mean bounds and standard deviations. We use a piecewise linear function and an exponential function as the generator to derive the safe approximation for the former case, whereas a piecewise quadratic function is used for the latter case. The approximations which we get from family-i are quadratic programming problems and the approximation for family-ii is a quadratically constrained quadratic programming problem. Thus these approximations are computationally tractable and by solving them we get the robust solutions to our original ambiguous chance constrained problems. Moreover, we solve a numerical problem by the use of these approximations and we plot the optimal allocations and efficient frontiers for each case. From the two approximations of family-ii, it is observed that the optimal portfolio risks for the Bernstein approximation are higher than that of the piecewise linear generator based approximations.\\
Furthermore, our study can be used for any optimization problem with the linear chance constraint(s). In this study, we use piecewise linear and quadratic functions as the generators for deriving the safe approximations. For the derivation of piecewise linear generator based approximations, we require the information regarding the mean bounds of the perturbations, whereas for piecewise quadratic generator based approximations we need the standard deviation (or variance) information of the perturbations along with their mean bounds. Similarly, if we have the information regarding the first $n$ moment statistics of the perturbations, then we can extend our study to use piecewise generating functions of any degree $n$, where $n$ is a positive integer.\\ 
\\
\textbf{Data Availability:} The Nifty price data of the three sectors-- Nifty Bank, Nifty Infra and Nifty IT used in the study are collected for the time period June 2017 to May 2022 and the data is available in the website \href{https://finance.yahoo.com}{https://finance.yahoo.com}. We also attach the data in the file named Nifty price data\\
\textbf{Conflict of Interest Statement:}
The authors declare that they have no conflicts of interest.

%
%


\newpage
\begin{thebibliography}{}
	%
	\bibitem{Markowitz J 1952}
	Markowitz, H., Portfolio Selection, Journal of Finance, Vol. 7, pp. 77-91 (1952).
	\bibitem{Markowitz 1959}
	Markowitz, H., Portfolio Selection: efficient diversification of investments, Basil Blackwell, New York (1959).
	\bibitem{Estrada 2002}
	Estrada, J. Systematic risk in emerging markets: the D-CAPM. Emerging Markets Review, 3(4), pp. 365-379 (2002).
	\bibitem{Estrada 2007}
	Estrada, J. Mean-semivariance behavior: Downside risk and capital asset pricing. International Review of Economics \& Finance, 16(2), pp. 169-185 (2007).
	\bibitem{Ghaoui 1997}
	El Ghaoui, L., and Lebret, H. Robust solutions to least-squares problems with uncertain data. SIAM Journal on matrix analysis and applications, 18(4), pp. 1035-1064 (1997).
	\bibitem{Ghaoui 1998}
	El Ghaoui, L., Oustry, F., and Lebret, H. Robust solutions to uncertain semidefinite programs. SIAM Journal on Optimization, 9(1), pp. 33-52 (1998).
	\bibitem{Ben-Tal 1999}
	Ben-Tal, A., and Nemirovski, A.  Robust solutions of uncertain linear programs. Operations research letters, 25(1), pp. 1-13 (1999).
	\bibitem{Ben-Tal 2000}
	Ben-Tal, A., and Nemirovski, A.  Robust solutions of linear programming problems contaminated with uncertain data. Mathematical programming, 88(3), pp. 411-424 (2000).
	\bibitem{Goldfarb 2003}
	Goldfarb, D., and Iyengar, G., Robust portfolio selection problems. Mathematics of operations research, 28(1), pp. 1-38 (2003).
	\bibitem{Tutuncu 2004}
	T\"{u}t\"{u}nc\"{u}, R. H., and Koenig, M., Robust asset allocation. Annals of Operations Research, 132(1-4), pp. 157-187 (2004).
	\bibitem{Fliege 2014}
	Fliege, J., and Werner, R. Robust multiobjective optimization \& applications in portfolio optimization. European Journal of Operational Research, 234(2), pp. 422-433 (2014).
	\bibitem{Sehgal 2019}
	Sehgal, R. and Mehra, A. Robust reward-risk ratio portfolio optimization. International Transactions in Operational Research (2019).
	\bibitem{Biswal 1998}
	Biswal, M. P., Biswal, N. P., and Li, D. (1998). Probabilistic linear programming problems with exponential random variables: A technical note. European Journal of Operational Research, 111(3), 589-597.
	\bibitem{Calafiore 2006}
	Calafiore, Giuseppe Carlo, and L. El Ghaoui, "On distributionally robust chance-constrained linear programs." Journal of Optimization Theory and Applications 130, no. 1, pp. 1-22 (2006).
	\bibitem{Erdogan 2006}
	Erdoğan, E., and Iyengar, G., Ambiguous chance constrained problems and robust optimization. Mathematical Programming, 107(1), pp. 37-61 (2006).
	\bibitem{Nemirovski 2007}
	Nemirovski, A., and Shapiro, A., Convex approximations of chance constrained programs. SIAM Journal on Optimization, 17(4), pp. 969-996 (2007).
	\bibitem{Yanıkoğlu 2013}
	Yanıkoğlu, İ., and den Hertog, D. Safe approximations of ambiguous chance constraints using historical data. INFORMS Journal on Computing, 25(4), pp. 666-681 (2013).
	\bibitem{Bertsimas 2018}
	Bertsimas, D., Gupta, V., and Kallus, N., Data-driven robust optimization. Mathematical Programming, 167, pp. 235-292 (2018).
	\bibitem{Zhang 2018}
	Zhang, Y., Jiang, R., and Shen, S. Ambiguous chance-constrained binary programs under mean-covariance information. SIAM Journal on Optimization, 28(4), pp. 2922-2944 (2018).
	\bibitem{Bertsimas 2011}
	Bertsimas, D., Brown, D. B., and Caramanis, C. Theory and applications of robust optimization. SIAM review, 53(3), pp. 464-501 (2011).
	\bibitem{Pulak 2021}
	Swain, P., \& Ojha, A. K., Bi-level optimization approach for robust mean-variance problems. RAIRO-Operations Research, 55(5), pp. 2941-2961 (2021).
	\bibitem{Ben-Tal 2009}
	Ben-Tal, A., El Ghaoui, L., and Nemirovski, A. Robust optimization (Vol. 28). Princeton University Press (2009).
	\end{thebibliography}


\end{document}